\documentclass[12 pt]{article}
\usepackage{amscd}
\usepackage{ytableau}
\usepackage{relsize}
\usepackage[font=small]{caption}
\usepackage{hyperref}
\usepackage{appendix}
\usepackage{amsfonts}
\usepackage{amsmath}
\usepackage{cleveref}
\usepackage{amssymb}
\usepackage{amsthm}
\usepackage{bbm}
\usepackage{diagbox}
\usepackage{esint}
\usepackage[export]{adjustbox}
\usepackage[margin = .75 in]{geometry}
\usepackage{graphicx}
\usepackage{listings}
\usepackage{marvosym}
\usepackage{mathrsfs}
\usepackage{mathtools}
\usepackage{multicol}
\usepackage{nicefrac}
\usepackage{pgf}
\usepackage{polynom}
\usepackage{stackengine}
\usepackage{tikz}
\usepackage{tikz-cd}
\usepackage{units}
\usepackage[normalem]{ulem}
\usepackage[utf8]{inputenc}
\usepackage[backend=biber,style=alphabetic]{biblatex}
\usepackage[]{youngtab}
\usepackage{ wasysym }
\usepackage{comment}
\stackMath

\theoremstyle{plain}
\newtheorem{theorem}{Theorem}[section]
\theoremstyle{plain}
\newtheorem*{theorem*}{Theorem}
\theoremstyle{plain}
\newtheorem{corollary}{Corollary}[section]
\theoremstyle{plain}
\newtheorem*{corollary*}{Corollary}
\theoremstyle{plain}

\theoremstyle{plain}
\newtheorem*{lemma*}{Lemma}
\theoremstyle{plain}
\newtheorem{proposition}{Proposition}[section]
\theoremstyle{plain}
\newtheorem*{proposition*}{Proposition}
\theoremstyle{plain}
\newtheorem{conjecture}[theorem]{Conjecture}
\theoremstyle{plain}
\newtheorem*{conjecture*}{Conjecture}
\theoremstyle{remark}
\newtheorem{remark}{Remark}[section]
\theoremstyle{remark}

\newtheorem*{remark*}{Remark}
\theoremstyle{definition}
\newtheorem{definition}{Definition}[section]
\theoremstyle{definition}
\newtheorem*{definition*}{Definition}
\theoremstyle{definition}
\newtheorem{example}{Example}[section]
\theoremstyle{definition}
\newtheorem*{example*}{Example}
\theoremstyle{definition}

\theoremstyle{definition}
\newtheorem*{question*}{Question}
\usetikzlibrary{arrows}
\usetikzlibrary{backgrounds}
\usetikzlibrary{calc}
\usetikzlibrary{cd}
\usetikzlibrary{decorations.pathmorphing}
\usetikzlibrary{fit}
\usetikzlibrary{petri}
\usetikzlibrary{positioning}
\usetikzlibrary{trees}
\graphicspath{ {} }
\DeclareGraphicsExtensions{.jpeg,.jpg,.png,.pdf}

\DeclareMathOperator{\C}{\mathbb{C}}

\DeclareMathOperator{\Hom}{Hom}

\DeclareMathOperator{\sgn}{sgn}

\DeclareMathOperator{\Stab}{Stab}

\DeclareMathOperator{\stab}{\text{Stab}}

\DeclarePairedDelimiterX{\norm}[1]{\lVert}{\rVert}{#1}

\newcommand{\NSf}{\textcolor{red}{/}}
\newcommand{\Df}{\textbf{\textcolor{blue}{\textbackslash}}}
\newcommand{\bT}{\mathsf{T}}
\newcommand{\bA}{\mathsf{A}}

\newcommand{\chamb}{\mathfrak{C}}

\renewcommand{\theconjecture}{%
  \ifnum\value{section}=0
    \thechapter.%
  \else
    \thesection.%
  \fi
  \arabic{theorem}%
}

\title{Integrals of stable envelopes for cotangent bundles to Grassmannians}
\author{Matthew Crawford \and Pavan Kartik \and Reese Lance}

\begin{document}
\maketitle

\abstract{We consider cohomological stable envelopes for a natural torus action $\bT$ on $X_{k,n}=T^*Gr(k,n)$, introduced by Maulik-Okounkov. We define the $\C^*_\hbar$-equivariant integral of the stable envelope using equivariant localization over the subtorus $\C^*_\hbar\subset\bT$, and compute the integral as a non-equivariant limit of the localization over the full torus, $\bT$. The integral of such a class is an integer times a power of $\hbar$, and the main result of this paper is a combinatorial formula for these integers. In 3d mirror symmetry, these non-equivariant limits are expected to reflect some curve counting phenomena on the 3d mirror dual, $X_{k,n}^\vee$. When $k=1$, we obtain the binomial coefficients, and we study some of the combinatorics of the integers for higher $k$, which have not appeared in the literature before. We give some conjectures and interpretations on extending this phenomena to type A quiver and bow varieties. }
\tableofcontents

\section{Overview}

\subsection{Stable Envelopes}

Cohomological stable envelopes were introduced by Maulik-Okounkov \cite{MO} in order to construct geometric actions of the Yangian algebra on the cohomology of Nakajima quiver varieties. These quiver varieties provide a rich class of symplectic resolutions of singularities, and play a key role in the 3D mirror symmetry conjectures. In this paper we restrict our attention to the most fundamental Nakajima quiver variety, the cotangent bundle to the Grassmannian. 

If a torus $\bT$ acts on a smooth variety, $X$, cohomological stable envelopes are maps of $H_\bT^\bullet(\text{pt})$-modules 

$$
\stab_{\chamb}: H_{\bT}^\bullet(X^\bT)\to H_{\bT}^\bullet(X)
$$
depending on a choice of chamber, $\chamb$, which is a connected component of the complement of all root hyperplanes in an associated root system. The stable envelope of a torus fixed point, $p\in X^\bT$, is the image of its equivariant cohomology class under this map, $\stab_{\chamb}(p)$. According to the theory of equivariant localization via \cite{AB}, each $\stab_\chamb(p)$ class is determined by its restrictions to fixed points. Thus we may think of the map $\stab_{\chamb}$ as being encoded in the matrix of restrictions 

$$
(\stab_\chamb)_{ij}=\stab_{\chamb}(p_i)\big|_{p_j}.
$$
In the enumerative formulation of 3D mirror symmetry, the elliptic stable envelopes conjecturally relate certain enumerative counts on $X$ to enumerative counts on its mirror dual, $X^!$, see \cite{AO}. When setup appropriately this ``relation" is via matrix multiplication and simple algebraic changes of variables. Mathematically these counts are given by vertex functions, equivariant curve counts, which are also of physical interest for the corresponding quiver gauge theories.

\subsection{Integration and Localization}

For the rest of the paper, we set $X_{k,n}=T^*Gr(k,n)$ and $\bT=(\C^*)^n\times\C^*_\hbar$. We define the $\mathbb{C}^{*}_{\hbar}$-equivariant integral, denoted
$$
\int_{X_{k,n}}\Stab(p)
$$
by viewing $\Stab(p)$ as a class in $H_{\C^*_\hbar}^*(X_{k,n})$, and performing equivariant localizing with respect to the torus $\C^*_\hbar$.

In this case, the integral is expressed as a sum of rational functions in the parameter $\hbar$ representing the coordinate on $\C^*_{\hbar}$. We prove in Proposition (\ref{int-is-integer}) that the integral\footnote{The scaling factor of $\hbar$ is just a convenience for formulas.}
$$
\int_{X_{k,n}}\hbar^{k(n-k)}\cdot \stab(p_I)
$$
is actually an integer.\\
We compute this integral in 2 steps: We use $\bT-$equivariant localization, i.e we compute the sum 
\begin{equation}\label{stabsum}
    \sum\limits_{p_{J}\in X_{k,n}^{\bT}}\frac{\stab(p)|_{p_{J}}}{e(T_{p_{J}}X_{k,n})}\end{equation}
whose terms are rational functions in $\mathbf{a}:=(a_{1},\dots,a_{n})$ (the coordinates on $(\mathbb{C}^{*})^{n}\subset\bT$) and $\hbar$. We then compute the limit of the quantity in (\ref{stabsum}) as $\textbf{a}\to 0$, and we prove in Proposition (\ref{comm-diagram}) that this limit equals $\int_{X_{k,n}}\stab(p)$. 

Our study is motivated by the role of stable envelopes in 3d mirror symmetry, and its interaction with vertex functions. Originally, the integrals of stable envelopes as we defined appear as coefficients of expansion of the integral representation of vertex functions for $T^{*}Gr(k,n)$, see \cite{SV2} and \cite{SV}. 

In \cite{DSm}, they compute non-equivariant limits of vertex functions on Nakajima quiver varieties, and conjecture that the result is related to representation-theoretic information on the corresponding 3d mirror dual. In \cite{DSm2}, they compute Euler characteristics of stable envelopes for the cotangent bundle to the full flag variety, and relate it to an enumerative object on the 3d mirror dual variety called the index vertex. 

\subsection{Nonequivariant limits}

The sums appearing in (\ref{stabsum}) are, for example,

\begin{example}
Consider the case $X_{1,4}=T^*\mathbb P^3$, and the corresponding torus, $\bT=(\C^*)^4$. There are 4 fixed points, $p_1,\dots,p_4$. Combining terms in the sum, we have

\begin{align*}
\sum\limits_{p_{j}\in X_{1,4}^{\bT}}\frac{\stab(p_{1})|_{p_{j}}}{e(T_{p_{j}}X_{1,4})}
&=-\frac{1}{(-\hbar+a_{1}-a_{2})(-\hbar+a_{1}-a_{3})(-\hbar+a_{1}-a_{4})}\\
\\
\sum\limits_{p_{j}\in X_{1,4}^{\bT}}\frac{\stab(p_{2})|_{p_{j}}}{e(T_{p_{j}}X_{1,4})}
&=\frac{3\hbar^{2}-3\hbar a_{1}-a_{1}^{2}-\hbar a_{2}+2\hbar a_{3}-a_{1}a_{3}+2\hbar a_{4}-a_{1}a_{4}+a_{3}a_{4}}{(\hbar+a_{2}-a_{1})(\hbar+a_{3}-a_{1})(\hbar+a_{3}-a_{2})(\hbar+a_{4}-a_{3})(\hbar+a_{4}-a_{2})}\\
\\
\\
\sum\limits_{p_{j}\in X_{1,4}^{\bT}}\frac{\stab(p_{3})|_{p_{j}}}{e(T_{p_{j}}X_{1,4})}
&=\frac{3\hbar^{2}-2\hbar a_{1}-2\hbar a_{2}+a_{1}a_{2}+\hbar a_{3}+3\hbar a_{4}-a_{1}a_{4}-a_{2}a_{4}+a_{4}^2}{(\hbar+a_{3}-a_{1})(\hbar+a_{3}-a_{2})(\hbar+a_{4}-a_{1})(\hbar+a_{4}-a_{2})(\hbar+a_{4}-a_{3})}\\
\\
\\
\sum\limits_{p_{j}\in X_{1,4}^{\bT}}\frac{\stab(p_{4})|_{p_{j}}}{e(T_{p_{j}}X_{1,4})}
&=\frac{1}{(\hbar-a_{1}+a_{4})(\hbar+a_{4}-a_{2})(\hbar+a_{4}-a_{3})}.
\end{align*}

\end{example}
The non-equivariant limits, which are the highest power of $\hbar$ in the numerator, are exactly the binomial coefficients $(1,3,3,1)$ (times a power of $\hbar)$. Therefore we are led to guess that the ${n\choose 1}=n$ integers corresponding to each $(\C^*)^{n-1}$-fixed point of $T^*\mathbb{P}^{n-1}$ are the binomial coefficients
\[
\int_{X_{1,n}}\stab(p_i)=\binom{n}{i-1}\cdot \hbar^{n-1},
\]
and to also wonder what the generalization to $X_{k,n}$ may be. The answer must be some system of integers which generalize the binomial coefficients in some way, which is interesting independent of the geometric setting. The main content of this paper is to present a proof of a closed form for these integrals for the case of $X_{k,n}$. 

\begin{theorem*}[Theorem \ref{lim-calc}]

We give an explicit, combinatorial formula $F(p_I)=\int_{X_{k,n}}\Stab(p_I)$, in terms of only $I,n,k$. 
\end{theorem*}

The formula $F$ is a rational function, and it is not clear that, upon substituting in specific values of $I,k,n$, it becomes an integer. Nevertheless this must be true, again because it is equal to $\int_{X_{k,n}}\Stab(p_I)$.

When setting $k=1$, $F$ recovers the binomial coefficients. For higher $k$ we find sequences of integers with interesting combinatorial relations. Many statements we will make about the combinatorics of these integers remain as conjectures, and we pose further questions for which we cannot even make conjectures. 

\subsection{Combinatorics of integrals}

The $n$ integers $\int_{X_{1,n}}\Stab(p)$ (ignoring the $\hbar$ factor) obtained from $X_{1,n}=T^*\mathbb{P}^{n-1}$ are the binomial coefficients, and thus can be arranged into Pascal's triangle, with its recurrence relation given by Pascal's rule. This is an infinite triangle (2-simplex) where each number in a given layer is the sum of the two numbers in the layer directly above it. We conjecture that a similar phenomenon holds for the $\binom{n}{k}$ integers arising from integrals over $X_{k,n}$ (after the addition of some correction terms) arranged in a $(k+1)-$ simplex, where the integers in one layer are determined by the $2^{k}$ ``nearby" integers in the layer above (we will refer to this phenomenon as ``$2^{k}$-neighbor addition"). We prove this statement for the case $k=2$. We demonstrate another method to modify the combinatorial problem so that 4-neighbor addition holds without the addition of the ``correction" terms. These combinatorial structures may be shadows of nearby geometric problems. 

These results we conjecture for higher $k$ (and prove for $k=2$) demonstrate the complexity of this family of integers. This is especially interesting when they are viewed as a generalization of the binomial coefficients. 

We also provide an outline of a new conjectural method to find the $\bT$-equivariant integrals of stable envelopes due to A. Varchenko. We introduce the quantities $\mathbb{V}({\lambda})$, for each partition $\lambda$ in an $(n-k)\times k$ rectangle, which are some sums of weighted paths to the Young diagram $Y_\lambda$. The weights are in formal variables $\textbf{a},\hbar$, and we conjecture that $\mathbb{V}(\lambda)$ is exactly the $\bT$-equivariant integral of the corresponding fixed point. 

\subsection{Future Directions}
Cohomological stable envelopes may be defined naturally for all type A quiver varieties. Though our definition via $\C^*_\hbar$-equivariant localization no longer makes sense, their integrals can nevertheless be computed by $\bT$-equivariant localization. We conjecture their nonequivariant limits also exist and are rational numbers times some power of $\hbar$, supported by computer experiments.

For bow varieties, we find that the non-equivariant limit mentioned above does not necessarily exist, though we do not have a conceptual reason for why that is the case. Such an explanation would be interesting to learn. 

One interesting property we observe is the following, (see section \ref{sec5} for precise statement):

\begin{conjecture*}[\ref{bow}]
A bow variety $\mathcal{C}$ is Hanany-Witten equivalent to a quiver variety
if and only if the non-equivariant limit of the integral of the stable envelope $\displaystyle \lim_{\textbf{a}\to 0}\int_{\mathcal{C}}\Stab(p)$ exists for all torus fixed points $p$. 
\end{conjecture*}

In this paper, we study cohomological stable envelopes. In the same settings there are also K-theoretic and elliptic stable envelopes and weight functions and vertex functions, and we plan to extend these results to those cases.

\subsection{Acknowledgements}
We thank Andrey Smirnov and Alexander Varchenko for introducing this problem to us. We further thank Corbin Balitactac, Tommaso Botta, Hunter Dinkins, Jacob Folks, and Rich\'ard Rim\'anyi for helpful discussions and suggestions. In particular, some of the conjectures we make here are supported by computer experiments performed with Maple code written by Rich\'ard Rim\'anyi and Andrey Smirnov. M. Crawford was partially supported by the NSF Grant DMS-1901796. P. Kartik was partially supported by the NSF grant DMS-2401380.

\section{Setup}
\subsection{Stable Envelopes}\label{stabsection}
We summarize the axiomatic treatment of stable envelopes as in \cite{MO}, in our restricted setting. For a more general and detailed discussion, see the cited source. \\\\
Let $X_{k,n}:=T^*Gr(k,\C^n)$ be the cotangent bundle to the Grassmannian parameterizing $k$-dimensional linear subspaces of $\C^n$. Let $\bA:=(\C^*)^n$, which has an induced action on $X_{k,n}$ by the action on the underlying $\C^n$ that preserves the symplectic form. The larger torus $\bT:=\bA\times \C^*_\hbar$ acts on $X_{k,n}$ by letting $\hbar$ scale the cotangent fibers. The fixed points of the $\bT$ action all lie in the 0-section of $X_{k,n}$ and are labeled by increasing subsets $I\subset [n]$ of size $k$. The fixed point corresponding to the subset $I$ is denoted $p_I$. We will not distinguish between $p_I$ as an element of $T^*Gr(k,n)$ versus $Gr(k,n)$.

The inclusion $i:X_{k,n}^\bT\to X_{k,n}$ induces a restriction on $\bT$-equivariant cohomology
$$
i^*: H_\bT^\bullet(X_{k,n})\to H_\bT^\bullet(X_{k,n}^\bT).
$$
Stable envelopes are  maps in the other direction
$$
\stab_{\mathfrak C}: H^\bullet_\bT(X_{k,n}^\bT)\to H^\bullet_\bT(X_{k,n})
$$
satisfying a collection of axioms, which we must fix some notation to state: We fix the cocharacter $\sigma: \C^*\to\bA$
$$
z\mapsto (z,z^2,\dots,z^n)
$$
and the corresponding ``standard" chamber as $\mathfrak C=\{1,2,\dots,n\}$, meaning $a_1\leq a_2\leq \dots \leq a_n$. These choices are constant throughout, so we omit the usual subscripts which indicate a dependence on the chamber. We only consider the standard chamber because any other choice of chamber will simply permute which integral is associated to which fixed point, along with appropriate substitutions in the $a$-variables: We state this proposition here, which is not from the Maulik-Okounkov source, but reserve the proof for Section \S\ref{weight-fcn-section}. Any other chamber, $\chamb^\tau$, is related to $\chamb$ by a permutation $\tau\in S_n$. A choice of chamber is denoted by a permutation $[\nu_{1},\dots,\nu_{n}]$ of $1,\ 2,\ \dots,n$, and $[\nu_{1},\dots,\nu_{n}]$ denotes the chamber in which $a_{\nu_{1}}\leq a_{\nu_{2}}\leq\dots a_{\nu_{n}}$. In this case $[\nu_{1},\dots,\nu_{n}]^{\tau}$ denotes the chamber $[\tau(\nu_{1}),\dots,\tau(\nu_{n})]$.

\begin{proposition*}[Proposition (\ref{chamber-dependence})]
Let $\tau \in S_n$. Then 
\[
\int_{X_{k,n}}\Stab_{\chamb^{\tau}}(p) = \tau\left( \int_{X_{k,n}} \Stab_{\chamb}(\tau^{-1}(p)) \right)
\]
where $\tau(p) = \{\tau(i)\ |\ i \in p\}$, and the action of $\tau$ on the left hand side sends $a_i \mapsto a_{\tau(i)}$.
\end{proposition*}

Due to this proposition, we assume the standard chamber for the rest of our calculations. The choice of chamber determines a $\bT$-invariant decomposition of the tangent space of any fixed point into attracting and repelling directions
$$
T_I{X_{k,n}} \cong N_I^+ \oplus N_I^-.
$$
Denote the attracting set of a fixed point $p_I\in X_{k,n}^\bT$
$$
\text{Attr}(p_I)=\big\{x\ \big|\ \lim_{z\to 0}(z,z^2,\dots,z^n)\cdot x=p_I\big\},
$$
and the full attracting set
$$
\text{Attr}^f(p_I) = \bigsqcup_{p_J\prec p_I} \text{Attr}(p_J),
$$
and say that $p_J\prec p_I$ if $p_J \in \text{Attr}_{\chamb}^f(p_I)$. The transitive closure of this relation defines a partial order on the set of torus fixed points.

\begin{theorem}[\cite{MO}]\label{stab-defn}
There exists a unique map of $H_\bT^*(\text{pt})$-modules
$$
\Stab: H_\bT^*(X_{k,n}^\bT)\to H_\bT^*(X_{k,n})
$$
such that for any $p_I \in X_{k,n}^\bT$ the stable envelope $\Gamma=\Stab([p_I])$ satisfies: \\\\
i) $\mathrm{supp}\ \Gamma\subset \mathrm{Attr}^f(p_I)$,\\
ii) $\Gamma|_{p_I}=e(N_-)$,\\
iii) $\mathrm{deg}_\bA\Gamma\big|_{p_J}<\frac12\mathrm{dim}(X_{k,n})$, \quad for any $p_J\prec p_I$, $p_J\ne p_I$. 
\end{theorem}

\begin{example}
When $k=1,\ n=3$ i.e $X_{1,4}=T^*\mathbb{P}^3$, the $4\times 4$ matrix $(\stab_\chamb)_{ij}$ is
\[\begin{pmatrix}
   (12)_\hbar(13)_\hbar(14)_\hbar & \hbar(23)_\hbar(24)_\hbar & \hbar(32)_\hbar(34)_\hbar& \hbar(42)_\hbar(43)_\hbar\\
   
   0 & (12)(23)_\hbar(24)_\hbar & \hbar(13)(34)_\hbar & \hbar(14)(43)_\hbar\\
   
   0 & 0 & (13)(23)(34)_\hbar & \hbar(14)(24)\\
   
   0 & 0 & 0 & (14)(24)(34)
\end{pmatrix}\]
where for readability we abbreviate $(ij)\equiv (a_i-a_j)$ and $(ij)_\hbar\equiv (a_i-a_j+\hbar)$. Here $a_i$ are the coordinates of the torus $\bA=(\C^*)^4$ and $\hbar$ is the coordinate of an additional 1-dimensional torus, $\C^{*}_{\hbar}$ scaling the cotangent fibers. 

This matrix can be computed using the defining axioms of stable envelopes as in Theorem (\ref{stab-defn}), and it also follows immediately after accepting Theorem (\ref{weight-fcns-stab}). 
\end{example}

\subsection{Equivariant Localization and Integration}\label{eq-loc}

Over $X_{k,n}$ there is an exact sequence of complex vector bundles
$$
0\to \mathcal S \to \C^n\to \mathcal Q \to 0.
$$
$\mathcal S$ and $\mathcal Q$ are the tautological sub and quotient bundles of rank $k$ and $n-k$, respectively, while $\C^n$ is the trivial bundle. Denote the Chern roots of $\mathcal S, \mathcal Q$, and the standard representation of $\bT$ as $t_1,\dots,t_k$, $w_1,\dots,w_{n-k}$, and $a_1,\dots,a_n$, respectively.  It is well known that the equivariant cohomology of $X_{k,n}$ is generated by the collection of these classes. 

Similarly well-known is the so-called GKM description of the equivariant cohomology due to \cite{GKM}

\[
H_\bT^\bullet(X_{k,n})\cong\Big\{ (f_x)\in\bigoplus_{x\in X_{k,n}^\bT}H^\bullet_\bT(\{x\})\ \Big|\ \chi_E\ |\ (f_x-f_y) \Big\}
\]
where $\chi_E$ is the character labeling the directed edge $E: x \to y$ in the moment graph if there is such an edge. We provide a combinatorial description of the moment graph in \S\ref{mgraph}, which one may use as an aid to compute stable envelope matrices.

The projection 

$$
\pi: X_{k,n} \to \{\text{pt}\}
$$
is not proper, so there is no push-forward in cohomology. In such case, it is common to \textit{define} the integral via equivariant localization. Let $\widetilde{\Stab}(p)\in H^*_{\C^{*}_{\hbar}}(X_{k,n})$ denote the image of $\Stab(p)$ under the map
\[
j:H^{*}_{\C^{*}_{\hbar}}(X_{k,n})\to H^{*}_{\bT}(X_{k,n})
\]
induced by inclusion of the subtorus $\C^{*}_{\hbar}\hookrightarrow\bT$ and $\text{Id}:X_{k,n}\to X_{k,n}$, which is equivariant with respect to the torus inclusion. We can then define the integral via equivariant localization over $\C^{*}_{\hbar}$. The $\C^{*}_{\hbar}$-fixed locus is $Gr(k,n)$, so 

\begin{equation}\label{stabp-int-def}
\int_{X_{k,n}}\Stab(p):= \int_{Gr(k,n)}\frac{\widetilde{\Stab}(p)|_{Gr(k,n)}}{e(N(Gr(k,n))} = \int_{Gr(k,n)}\frac{\widetilde{\Stab}(p)|_{Gr(k,n)}}{e(X_{k,n})}
\end{equation}
where $N(-)$ denotes the normal bundle to $(-)$, which in this case is nothing but $X_{k,n}$, viewed as a bundle over $Gr(k,n)$. This integral is finite because $Gr(k,n)$ is compact, which is the motivation for localizing over $\C^*_\hbar$. In fact, this integral must be an integer times some power of $\hbar$. First we prove it is an integral power series in $\hbar^{-1}$.

\begin{proposition}\label{int-is-integer}
We have
\[
\int_{X_{k,n}}\Stab(p)\in\mathbb{Z}[[\hbar^{-1}]].
\] 
\end{proposition}

\begin{proof}
Beginning with definition (\ref{stabp-int-def}), we have
\begin{align*}
\frac1{e(X_{k,n})}&=\frac1{\prod\limits_{i=1}^{n}(\hbar+y_{i})}=\frac{1}{\hbar^{n}}\prod\limits_{i=1}^{n}\frac{1}{1+\frac{y_i}{\hbar}}\\
&=\frac{1}{\hbar^{n}}\sum\limits_{j=0}^{\infty}(-1)^{j}\hbar^{-j}\mathbf{h}_{j}(y_{1},\dots,y_{n})\in H^{*}(Gr(k,n);\mathbb{Z})[[\hbar^{-1}]]
\end{align*} 
where $y_{i}$ denote the Chern roots of the vector bundle $T^*Gr(k,n)$, and $\mathbf{h}_{j}$ denotes the $j$-th complete symmetric polynomial.

Since $\mathbf{h}_{j}(y_{1},\dots,y_{n})$ and $ \widetilde{\Stab}(p)|_{Gr(k,n)}\in H^{*}(Gr(k,n);\mathbb{Z})[[\hbar^{-1}]]$, the integrand in (\ref{stabp-int-def}) also lies in $H^{*}(Gr(k,n);\mathbb{Z})[[\hbar^{-1}]]$. Since integrating a class in $H^{*}(Gr(k,n);\mathbb{Z})$ yields an integer, the result is shown.
\end{proof}

It remains to show the result of the integral is actually a \textit{monomial} in $\hbar^{-1}$, which would complete the integrality claim.  

We choose to introduce further localization over $\bT$, the full torus acting on $X_{k,n}$, for two reasons: 1) By analyzing the $\bT$-localization formula for the integral, we can prove the integral is indeed a monomial in $\hbar^{-1}$, and 2) It provides an alternative method of computing the integral, as opposed to the method outlined in Proposition \ref{int-is-integer}. This choice introduces a tradeoff: The fixed locus is now finite, but features additional parameters, the coordinates of the torus $\bA$. To recover the original quantity, we take the limit as the ``auxiliary'' equivariant parameters go to 0. This is another way of computing the integral we have defined.

\begin{proposition}\label{comm-diagram}
We have
\begin{equation}\label{comm-diagram-eqn}
\int_{X_{k,n}}\Stab(p_I) = \sum_{p_J\in X_{k,n}^\bT}\frac{\Stab(p_I)|_{p_J}}{e(T_{p_J}X_{k,n})}\Bigg|_{a_i=0}.
\end{equation}

\end{proposition}

\begin{proof}
Follows from the commutativity of the diagram
\[
\begin{tikzcd}
    H^{*}_{\bT}(X_{k,n})\arrow[r,"p"]\arrow[d,"j"] &H^{*}_{\bT}(pt)_{loc}\arrow[d,"\widetilde{j}"]\\
    H^{*}_{\mathbb{C}^{*}_{\hbar}}(X_{k,n})\arrow[r,"p'"] &H^{*}_{\mathbb{C}^{*}_{\hbar}}(pt)_{loc}
\end{tikzcd}
\]
The map $\widetilde{j}$ is induced by the inclusion $\mathbb{C}^{*}_{\hbar}\hookrightarrow\bT$, and the maps $p$ and $p'$ come from $\bT$-equivariant localization and $\mathbb{C}^{*}_{\hbar}$-equivariant localization respectively.
\end{proof}

As claimed, with this proposition we can then show the integrality claim.

\begin{corollary}
We have
\[
\int_{X_{k,n}}\Stab(p)\in \mathbb Z\cdot \hbar^{k(n-k)}. 
\]
\end{corollary}

\begin{proof}
By Proposition \ref{int-is-integer}, the above integral lies in $\mathbb Z[[\hbar^{-1}]]$. Using $\bT$-equivariant localization, we show in the appendix that it must further be a monomial in $\hbar^{-1}$. Therefore it is an integer times a power of $\hbar$, see \ref{limit monomial}. By cohomological degree argument, that power must be $k(n-k)$. 
\end{proof}

Let us mention here that it is not obvious that the non-equivariant limit in the RHS of (\ref{comm-diagram-eqn}) exists, as $e(T_{p_J}X_{k,n})$ often has zeros when setting $a_i=a_j$ for $i,j\in [n]$. Taking the RHS of (\ref{comm-diagram-eqn}) in isolation, it can be proven directly that the non-equivariant limit exists by showing that all terms of the form $(a_i-a_j)$ in the denominator magically cancel with terms in the numerator when the sum is combined into a single term. This method describes a refactoring/resummation in the numerator, so it is a very nontrivial proof, which we include in Appendix A. However, the simplest way to see that the non-equivariant limits exist is by the fact that it equals $\int_{X_{k,n}}\Stab(p)$.

\subsection{Equivariant Structure of $X_{k,n}$}\label{mgraph}
    In this section we recall the equivariant structure of the $\bT$-action on $X_{k,n}$, which is standard material. An introductory source on these objects and computations can be found in \cite{AF}. 

    We first discuss the $\bA$-equivariant structure of $Gr(k,n)$, as the $\bT$-equivariant structure of $X_{k,n}$ is only a slight modification of this. For every $\bT$-fixed point, $p_I$, $T_{p_I}Gr(k,n)$ decomposes into one-dimensional $\bA$-weight spaces. The tangent space itself is of the form $\Hom(S,Q)$, where $S$ is the vector space representing $p_I$ and $Q$ is the quotient $\C^n/S$, thus by tensor-hom has a basis given by 
\[
\big\{e_s^*\otimes e_q \big|\ s\in I,q\not \in I\big\}
\]
where $\{e_i\}$ is the standard basis on the underlying $\C^n$. Then the $\bA$-weight in the direction of $e^*_s\otimes e_q$ is $a_q-a_s$, where $a_i$ are the coordinates of the torus $\bA$. Each $e_s^*\otimes e_q$ induces a pair of $\bA$-weight spaces in the cotangent bundle, $X_{k,n}$: $a_q-a_s$ and $a_s-a_q$. When considering $\bT$-weights, the cotangent directions are scaled by $\hbar$. In total, the $\bT$ characters of $T_{p_I}X_{k,n}$ are $a_q-a_s,a_s-a_q+\hbar$ for every pair $(s\in I,q\not\in I)$. Note that the characters are only determined up to sign, but we fix this choice throughout. 

Having found the $\bT$-weights of $T_{p_I}X_{k,n}$, we have an explicit formula for the equivariant Euler class appearing in \S\ref{eq-loc}.
\begin{proposition}
We have
\[
e(T_{p_I}X_{k,n}) = \prod\limits_{\substack{v=1\\v\not \in I}}^n (vi_1)(vi_2)\cdots(vi_k)(i_1v)_\hbar(i_2v)_\hbar\cdots (i_kv)_\hbar.
\]
\end{proposition}

Let us note here that the formula for the Euler class is independent of the choice of chamber. 

The fixed points, the 1-dimensional orbits containing fixed points, and $\bT$-characters along these lines, may be assembled together into a \textit{moment graph}. This is a graph arising from any nice variety with a torus action, whose vertices are the points $X_{k,n}^{\bT}$, with an edge between $p_1$ and $p_2$ if there exists a 1-dimensional $\bT$-orbit whose closure contains $p_1$ and $p_2$. This edge is labeled with the $\bT$-weight of the this orbit, as calculated above. And if a chamber is chosen, as we have, edges can also be directed according to the attracting partial order. 

The moment graph of $X_{k,n}$ has a simple description. Its vertices are given by the fixed points of $X_{k,n}$, $p_I$. There is an edge between $p_I$ and $p_J$ if $|I\cap J|=k-1$: Thus the symmetric difference consists of two indices, denote them as $i,j$. The $\bT$-weight at $T_{p_I}X_{k,n}$ in the direction of $p_J$ is $a_j-a_i$, and each fixed point has a cotangent direction of the opposite attracting nature with inverted $\bA$-weight and an additional additive $\hbar$ factor. With the chosen cocharacter, the induced partial order is determined by the indices $i,j$: The smaller index's corresponding fixed point is attracting. For example, in the moment graph of the $\bT$-action on $X_{2,4}$, the edge pointing from $\langle 13\rangle$ to $\langle 12\rangle$, for example, has character $a_3-a_2$ at $\langle 12\rangle$, and the corresponding arrow pointing away from $\langle 12\rangle$ has character $a_2-a_3+\hbar$. The $\bA$-weights are only determined up to sign, but the zero section weights and cotangent weights must have opposite signs.

\subsection{Weight functions}\label{weight-fcn-section}
Weight functions are rational functions introduced in \cite{TV} to describe certain solutions to quantum Knizhnik-Zamolodchikov equations. Later, \cite{RTV} showed that these weight functions may be interpreted as representatives for the stable envelope classes. The rest of this section summarizes the results we need from this paper, and gives the delayed proof from subsection \S\ref{stabsection}. 

The weight functions depend on a cocharacter or chamber, encoded by a permutation $\Sigma$, and we use $\sigma$ as the symmetrizing variable (our convention differs slightly from the source).

\begin{definition}
The \textit{weight function} for the standard chamber, $W(p_I)$, for $p_I\in X_{k,n}^\bT$ is 

\[
W(p_I):=\sum\limits_{\sigma \in S_k}\dfrac{\prod\limits_{r = 1}^k \prod\limits_{\alpha = 1}^{i_{r} - 1} (a_{\alpha} - t_{\sigma(r)}) \prod\limits_{\beta = i_r + 1}^n (t_{\sigma(r)} - a_{\beta} + \hbar) }{\prod\limits_{1 \leq \ell < m \leq k} (t_{\sigma(\ell)} - t_{\sigma(m)}) (t_{\sigma (\ell)} - t_{\sigma(m)} + \hbar)}
\]
and for another chamber $\chamb^\tau$ (in which $\{a_{\tau(1)}\leq a_{\tau(2)}\dots\leq a_{\tau(n)}\}$) is defined as
\[
W_{\chamb^\tau}(p_{I}):=\tau\left(W\left(p_{\tau^{-1}(I)}\right)\right)
\]
where $\tau:a_{i}\mapsto a_{\tau(i)},\ 1\leq i\leq n$.
\end{definition}

These are rational functions in variables $a,t,\hbar$, satisfying the property that, when the $t$ variables are evaluated at subsets of the $a$ variables corresponding to a fixed point $p_J$, $W_{\chamb^\tau}(p_I)|_{t_i=a_{J_i}}$, becomes a polynomial, and further that the collection 
\[
\big(W_{\chamb^\tau}(p_I)|_{t_i=a_{J_i}}\big)_{p_J\in X_{k,n}^\bT}
\]
defines a class in $H_\bT^*(X_{k,n})$, i.e. satisfies the GKM relations corresponding to the moment graph, denoted $[W_{\chamb^\tau}(p_I)]$.  

That weight functions are representatives of stable envelope classes is captured by the theorem

\begin{theorem}[\cite{RTV}]\label{weight-fcns-stab}
For all fixed points $I\subset [n]$, 
$$
[W_{\chamb^\tau}(p_I)]=\Stab_{\chamb^\tau}(p_I)
$$
where $[W_{\chamb^\tau}(p_I)]$ denotes the equivariant cohomology class of $W_{\chamb^\tau}(p_I)$. 
\end{theorem}
In the source, they have an additional factor $c_\lambda$ which accounts for different normalization. 

In particular, the stable envelope restriction to a fixed point can be computed by evaluating the weight function at the same fixed point,
\[
W_\Sigma(p_I)|_{p_J} = \Stab_\Sigma(p_I)|_{p_J}.
\]
With this theorem, we can prove the Proposition stated in Subsection \S\ref{stabsection} regarding the effect of choosing a different chamber, and the result is that a different chamber will simply permute which integral is associated to which fixed point, with $\bA$-parameter substitutions. In fact this proof shows that these integrals agree at the level of $\bT$-equivariant localization, before taking limits. 

\begin{proposition}\label{chamber-dependence}
Let $\chamb$ be the standard chamber, and $\tau \in S_n$ be a permutation representing another chamber, $\chamb^\tau$. Then
\[
\int_{X_{k,n}}\Stab_{\chamb^{\tau}}(p) = \tau\left( \int_{X_{k,n}} \Stab_{\chamb}(\tau^{-1}(p)) \right)
\]
where $\tau(p) = \{\tau(i)\ |\ i \in p\}$, and the action of $\tau$ on the left hand side sends $a_i \mapsto a_{\tau(i)}$.
\end{proposition}

\begin{proof}
For convenience, we swap $p\leftrightarrow \tau(p)$. Recall that 
\[
W_{\tau}(\tau(p)):=\tau(W(\tau^{-1}(\tau(p))))=\tau(W(p)).
\]
Computing the RHS by $\bT$-equivariant localization,
\[
\int_{X_{k,n}}\Stab_{\chamb^{\tau}}(\tau(p))=\sum_{p_{J}\in X_{k,n}^{\bT}}\frac{W_{\tau}(\tau(p))|_{p_{J}}}{e(T_{p_{J}}X_{k,n})}.
\]
We can write any $k$-subset $\{\theta_{1},\dots,\theta_{k}\}$ as $\{\tau(\tau^{-1}(\theta_{1})),\dots,\tau(\tau^{-1}(\theta_{k}))\}$\\\\
So we can index the sum by $\tau(p_{J})$) instead, and the integral can be written as
\[
\sum_{p_{J}\in X_{k,n}^{\bT}}\frac{(\tau(W(p)))|_{\tau(p_{J})}}{e(T_{\tau(p_{J})}X_{k,n})}.
\]
The LHS can be written similarly as
\[
\tau\left( \int_{X_{k,n}} \Stab_{\chamb}(p) \right)=\tau\left(\sum_{p_{J}\in X_{k,n}^{\bT}}\frac{W(p)|_{p_{J}}}{e(T_{p_{J}}X_{k,n})}\right).
\]
Applying $\tau$ to the denominator we obtain
\begin{align*}
\tau(e(T_{p_{J}}X_{k,n}))&=\tau\left(\prod\limits_{m=1}^{k}\prod\limits_{v\notin J}(a_{v}-a_{j_{m}})(a_{j_{m}}-a_{v}+\hbar)\right)\\&=\prod\limits_{m=1}^{k}\prod\limits_{v\notin J}(a_{\tau(v)}-a_{\tau(j_{m})})(a_{\tau(j_{m})}-a_{\tau(v)}+\hbar)\\
&=e(T_{\tau(p_{J})}X_{k,n}).
\end{align*}
Hence the LHS can be written as
\begin{align*}
\tau\left( \int_{X_{k,n}} \Stab_{\chamb}(p) \right)&=\sum_{p_{J}\in X_{k,n}^{\bT}}\frac{\tau(W(p)|_{p_{J}})}{e(T_{\tau(p_{J})}X_{k,n})}\\
&=\sum_{p_{J}\in X_{k,n}^{\bT}}\frac{(\tau(W(p)))|_{\tau(p_{J})}}{e(T_{\tau(p_{J})}X_{k,n})}\\
&=\sum_{p_{J}\in X_{k,n}^{\bT}}\frac{(\tau(W(p)))|_{p_{J}}}{e(T_{p_{J}}X_{k,n})},
\end{align*}
as required.
\end{proof}

\section{Calculation of Integral}\label{int-calc}

For convenience, we compute
\[
\hbar^{k(n-k)}\int_{X_{k,n}}\Stab(p_{I})
\]
so that the limit is just a number. By Proposition (\ref{comm-diagram}), the integral may equivalently be computed by localization over $\bT$ and taking $\textbf{a}\to 0$,
\[
\hbar^{k(n-k)}\int_{X_{k,n}}\Stab(p_{I})=\hbar^{k(n-k)}\sum_{p_J\in X_{k,n}^\bT}\frac{\Stab(p_I)|_{p_J}}{e(T_{p_J}X_{k,n})}\Bigg|_{a_i=0} 
\]
From Section \S\ref{weight-fcn-section}, we replace $\Stab(p_I)$ with the corresponding weight function, and the Euler classes are computed in \S\ref{mgraph}. Then
\begin{align*}\hbar^{k(n-k)}\int_{X_{k,n}}\Stab(p_{I})&=\hbar^{k(n-k)}\sum\limits_{p_{J}\in X_{k,n}^{\bT}}\frac{W(p_{I})|_{p_{J}}}{e(T_{p_{J}}X_{k,n})}\\
&=\sum_{p_J\in X_{k,n}^T}\dfrac{\sum\limits_{\sigma \in S_k} \dfrac{\prod\limits_{r = 1}^k \prod\limits_{\alpha = 1}^{i_{r} - 1} (a_{\alpha} - a_{j_{\sigma(r)}}) \prod\limits_{\beta = i_r + 1}^n (a_{j_{\sigma(r)}} - a_{\beta} + \hbar) }{\prod\limits_{1 \leq \ell < m \leq k} (a_{j_{\sigma(\ell)}} - a_{j_{\sigma(m)}}) (a_{j_{\sigma(\ell)}} - a_{j_{\sigma(m)}} + \hbar)}}{\prod\limits_{v\not \in J} \prod\limits_{p = 1}^{k}(a_v - a_{j_p}) (a_{j_p}-a_v+\hbar)}.
\end{align*}
We make the substitution $a_{s}\mapsto sz\hbar,\ 1\leq s\leq n$ to obtain
\begin{equation}\label{eq:integral}
\sum\limits_{p_{J}\in X_{k,n}^{\bT}}\frac{\sum\limits_{\sigma\in S_{k}}\frac{\prod\limits_{r=1}^{k}\prod\limits_{\alpha=1}^{i_{r}-1}(z(\alpha-j_{\sigma(r)}))\prod\limits_{\beta=i_{r}+1}^{n}(z(j_{\sigma(r)}-\beta)+1)}{\prod\limits_{1\leq l<m\leq k}(z(j_{\sigma(l)}-j_{\sigma(m)}))(z(j_{\sigma(l)}-j_{\sigma(m)})+1)}}{\prod\limits_{p=1}^{k}\prod\limits_{v\notin J}(z(v-j_{p}))(z(j_{p}-v)+1)}.
\end{equation}
It is sufficient to compute the limit $z\to 0$ of the expression above.\\\\
Fixed points of $X_{k,n}$ can be described as ordered tuples of distinct integers $1\leq i_{1}<\dots<i_{k}\leq n$. We introduce a partial order on fixed points by declaring $p_J\geq p_I$ if $j_r\geq i_r$ for all $r=1,\dots,k$. Then we can rewrite \eqref{eq:integral} as a single sum over tuples as
\[\sum\limits_{p_{J}\geq p_{I}}\frac{\prod\limits_{r=1}^{k}\prod\limits_{\alpha=1}^{i_{r}-1}(z(\alpha-j_{r}))\prod\limits_{\beta=i_{r}+1}^{n}(z(j_{r}-\beta)+1)}{\prod\limits_{1\leq l<m\leq k}(z(j_{l}-j_{m}))(z(j_{l}-j_{m})+1)\prod\limits_{p=1}^{k}\prod\limits_{v\notin J}(z(v-j_{p}))(z(j_{p}-v)+1)}.
\]
We can reduce to the case where the tuple $p_{J}\geq p_{I}$ since in the original expression if we have a permutation and a fixed point (thought of as a set) so that for some $r$ we have $j_{\sigma(r)}<i_{r}$, then $j_{\sigma(r)}\leq i_{r}-1$, which means that the first product in the numerator vanishes, and the term in the inner sum (over $\sigma\in S_{k}$) is 0. Hence the only terms that can possibly survive are for fixed points $\{j_{1},\dots j_{k}\}$ and permutations $\sigma\in S_{k}$ so that $j_{\sigma(r)}\geq i_{r}\ \forall 1\leq r\leq k$, which is equivalent to requiring that $p_{J}\geq p_{I}$ as tuples.  
After some algebraic manipulation we have

\[
\frac{(-1)^{|I|-k}}{z^{\binom{k+1}{2}+k(n-k)+|I|}\binom{|I|-k}{i_{1}-1,\ i_{2}-1,\dots,\ i_{k}-1}}\sum\limits_{p_{J}\geq p_{I}}\frac{\prod\limits_{r=1}^{k}\binom{j_{r}-1}{i_{r}-1}\prod\limits_{1\leq l<m\leq k}(1+(j_{m}-j_{l})z)}{\prod\limits_{r=1}^{k}\prod\limits_{v=j_{r}-i_{r}}^{j_{r}-1}(1+vz)\prod\limits_{1\leq l<m\leq k}(j_{l}-j_{m})\prod\limits_{p=1}^{k}\prod\limits_{v\notin J}(v-j_{p})}.
\]
Here the sum is over tuples $p_{J}$ with distinct entries, with each entry at most $n$, with $p_{J}$ entrywise greater than $p_{I}$, and $|I|:=i_{1}+\dots+i_{k}$ and $\binom{|I|-k}{i_{1}-1,\ i_{2}-1,\dots,\ i_{k}-1}$ denotes the multinomial coefficient.

The nonequivariant limit $a_{i}\to 0,\ 1\leq i\leq n$ can be obtained from the limit $z\to 0$. We now set $u:=\frac{1}{z}$ and take the limit $u\to\infty$. After rewriting $\Gamma(k+1)=k!$, $\frac{\Gamma(u+n+1)}{\Gamma(u)}=(u+n)(u+n-1)\dots u$ and further cancellation yields
\[(-1)^{|I|-k}u^{k(n-k+1)-|I|}\sum\limits_{p_{J}\geq p_{I}}\frac{(-1)^{|J|}\prod\limits_{1\leq l<m\leq k}(u+j_{m}-j_{l})(j_{m}-j_{l})}{\prod\limits_{r=1}^{k}\left(\Gamma(n-j_{r}+1)\Gamma(j_{r}-i_{r}+1)\right)}\prod\limits_{r=1}^{k}\frac{\Gamma(u+j_{r}-i_{r})}{\Gamma(u+j_{r})}.
\]
Since we are now taking the limit $u\to\infty$, we can employ the asymptotic expansion for the ratio of gamma functions given by \cite[\href{https://dlmf.nist.gov/5.11.iii}{5.11.(iii)}]{NIST:DLMF}
\[
\frac{\Gamma(u+a)}{\Gamma(u+b)}\sim u^{a-b}\sum\limits_{k=0}^{\infty}\frac{G_{k}(a,b)}{u^{k}}
\]
where 
\[
G_{k}(a,b)=\binom{a-b}{k}B_{k}^{(a-b+1)}(a).
\]
We also have a generating function for these generalized Bernoulli polynomials $B^{(l)}_{k}$ given by \cite{Temme:1995:BPO}

\[\left(\frac{t}{e^{t}-1}\right)^{l}e^{xt}=\sum\limits_{k=0}^{\infty}B^{(l)}_{k}(x)\frac{t^{k}}{k!}.
\]
We can now write each of the fractions $\frac{\Gamma(u+j_{r}-i_{r})}{\Gamma(u+j_{r})}$ as a power series in $\frac{1}{u}$. Since the $u\to\infty$ limit exists, and terms that have a negative exponent for $u$ go to 0 in the limit, the only term that contributes to the limit is the coefficient of $u^{0}$, which comes from the $u^{-k(n-k+1)+|I|}$ coefficient of 
\[\sum\limits_{J\geq I}\frac{(-1)^{|J|+|I|-k}\prod\limits_{1\leq l<m\leq k}(u+j_{m}-j_{l})(j_{m}-j_{l})}{\prod\limits_{r=1}^{k}\left(\Gamma(n-j_{r}+1)\Gamma(j_{r}-i_{r}+1)\right)}\prod\limits_{r=1}^{k}\frac{\Gamma(u+j_{r}-i_{r})}{\Gamma(u+j_{r})}.
\]
Using the definition of $G_{k}(a,b)$ we observe that 
\[G_{k}(a,b)=\left[\frac{x^{k}t^{k}y^{b-a-1}}{(b-a+k-1)!k!}\right]e^{y-x}\left(\frac{e^{t}-1}{t}\right)^{b-a-1}e^{at}.
\]
This can be seen via standard generating function arguments as in \cite{W}.
The product 
\[
\prod\limits_{1\leq l<m\leq k}(u+j_{m}-j_{l})
\]
can be expanded in terms of symmetric polynomials as 
\[\prod\limits_{1\leq l<m\leq k}(u+j_{m}-j_{l})=\sum\limits_{d=0}^{\binom{k}{2}}e_{\binom{k}{2}-d}(\Delta(J))u^{d}\] 
where 
\[
\Delta(J):=\{j_{m}-j_{l}|1\leq l<m\leq k\}.
\]
Hence the coefficient of $u^{-k(n-k+1)+|I|}$ can be calculated as 
\[\sum\limits_{J\geq I}\frac{(-1)^{|J|}\prod\limits_{1\leq l<m\leq k}(j_{m}-j_{l})}{\prod\limits_{r=1}^{k}(n-j_{r})!(j_{r}-i_{r})!}\times\left(\sum\limits_{d=0}^{\binom{k}{2}}e_{\binom{k}{2}-d}(\Delta(J))\times\left(\sum_{\substack{\lambda\vdash_{wc}\ k(n-k+1)+d-|I|\\ l(\lambda)=k}}\left(\prod\limits_{s=1}^{k}G_{\lambda_{s}}(j_{s}-i_{s},j_{s})\right)\right)\right)\]
where the sum \[\lambda\vdash_{wc} k(n-k+1)+d-|I|\] means that $\lambda$ is a weak composition of $k(n-k+1)+d-|I|$. From our description of $G_{p}(a,b)$, we can explicitly describe the terms $G_{\lambda_{s}}(j_{s}-i_{s},j_{s})$ in the innermost product:
\[
G_{\lambda_{s}}(j_{s}-i_{s},j_{s})=(i_{s}+\lambda_{s}-1)!\lambda_{s}![x^{\lambda_{s}}y^{i_{s}-1}t^{\lambda_{s}}]e^{y-x}\left(\frac{e^{t}-1}{t}\right)^{i_{s}-1}e^{t(j_{s}-i_{s})}
\]
\[
=(i_{s}+\lambda_{s}-1)!\lambda_{s}!\frac{(-1)^{\lambda_{s}}}{\lambda_{s}!}\times\frac{1}{(i_{s}-1)!}[t^{\lambda_{s}+i_{s}-1}]\left((e^{t}-1)^{i_{s}-1}e^{t(j_{s}-i_{s})}\right).
\]
An application of the binomial theorem yields
\[G_{\lambda_{s}}(j_{s}-i_{s},j_{s})=(-1)^{\lambda_{s}}\frac{1}{(i_{s}-1)!}\sum\limits_{p=0}^{i_{s}-1}(-1)^{i_{s}-1-p}(p+j_{s}-i_{s})^{\lambda_{s}+i_{s}-1}.
\]
Thus we obtain

\begin{theorem}\label{lim-calc}The following equality holds:

\begin{equation*}
    \lim_{\textbf{a}\to 0}\hbar^{k(n-k)}\left[\sum_{p_J\in X_{k,n}^\bT}\frac{\Stab(p_I)}{e(T_{p_J}X_{k,n})}\right]=(-1)^{k(n-k)-|I|}\sum\limits_{J\geq I}A_{J}B_{J}
    \end{equation*}
with
\[A_J := (-1)^{|J|}e_{\binom{k}{2}}(\Delta (J)) \prod _{r=1}^{k} \frac{1}{(n-j_r)!(j_r-i_r)!}\]
and
\[B_{J}:= \sum_{d=0}^{\binom{k}{2}} (-1)^d e_{\binom{k}{2}-d}(\Delta(J)) \sum_{\substack{\lambda\vdash_{wc}\ k(n-k+1)+d-|I|\\ l(\lambda)=k}}\left(\prod_{s = 1}^{k} \sum_{p = 0}^{i_s-1} \frac{(-1)^{p}(p + j_s - i_s )^{i_s + \lambda_s-1}}{(i_s-1-p)!p!}\right).
\]
\end{theorem}
The quantity
\begin{equation}\label{closed-form}
(-1)^{k(n-k)-|I|}\sum\limits_{J\geq I}A_{J}B_{J}
\end{equation}
is an explicit, combinatorial closed form for the integral, depending only on $I,n,k$. 

A priori, the expression in (\ref{closed-form}) is defined only for subsets $I$ of $\{1,\dots,n\}$ defining a fixed point of $X_{k,n}$. However (\ref{closed-form}) can also be evaluated at $k$-tuples $I=(i_{1}\leq\dots\leq i_{k})$ with $i_{1}=0, i_{k}=n+1$, or $i_{j}=i_{j+1}$ for some unique $j$, which we will utilize in Section \S\ref{4.1}. 
\begin{corollary}
By Proposition (\ref{comm-diagram}), we have
\[
\hbar^{k(n-k)}\int_{X_{k,n}}\Stab(p_I) = (-1)^{k(n-k)-|I|}\sum_{J\geq I}A_J B_J
\]
with the same notation.
\end{corollary}
For particular values of $k$, the generating function reduces further, though does still appear as a genuine rational function. Below we list two examples, though in principle this can be done for any $k$.
\begin{example}
When $k=1$ and $I=\{i\}$,
\[
\hbar^{n-1}\int_{X_{1,n}}\stab(p_i)=\binom{n-1}{i-1}
\]
recovering the binomial coefficients.
\end{example}
\begin{example}\label{gr2}
Setting $k=2$ and $I=\{i_{1}, i_{2}\}$, \eqref{closed-form} reduces to
\begin{multline*}
\hbar^{2(n-2)}\int_{X_{2,n}}\stab(p_{I})=
\frac{(n-2)!(n-3)!}{(i_{1}-1)!(i_{2}-1)!(n-i_{1})!(n-i_{2})!}\\\times(-3i_{1}+i_{2}-i_{1}^{2}+4i_{1}i_{2}-i_{2}^{2}+n(2(i_{1}-2i_{2}+1)+(i_{2}-i_{1})^{2})+n^{2}(i_{2}-i_{1})).
\end{multline*}
\end{example}
Even for $k=2$ it is not trivial to see that this quantity, when evaluated at integer values of $i_1,i_2,k,n$, is actually an integer.\\\\
Moreover, we observe that these quantities are positive integers.

\begin{conjecture}\label{positivity}
The integers $\int_{X_{k,n}}\Stab(p_{I})$ are positive for $p_I\in\left(X_{k,n}\right)^{\bT}$.
\end{conjecture}

\section{Combinatorial Interpretations}

We discuss the combinatorics of the ${n\choose k}$ integers arising from the $\C^*$-equivariant integrals of stable envelopes over $X_{k,n}$
\[
Z_{k,n} = \left\{\hbar^{k(n-k)}\int_{X_{k,n}}\Stab(p_I)\right\}_{p_I \in (X_{k,n})^\bT}
\]
as a finite set of integers. We find it constructive to fix $k$ and consider the infinite sequence $\cup_n Z_{k,n}$, along with its labeling of components in terms of $n$.  

As mentioned previously, $\cup_nZ_{1,n}$ is simply the set of binomial coefficients. An individual component $Z_{1,n}$ corresponds to the $n$th row of Pascal's triangle, so there is an obvious recurrence relation between $Z_{1,n}$ and $Z_{1,n+1}$ given by Pascal's rule. These integers for $k>1$ appear to be a new generalization of the binomial coefficients, so their combinatorics is of independent interest. 
\subsection{Variations of Pascal's $n$-simplex}\label{4.1}

The main conjecture of this section is a recurrence relation between $Z_{k,n}$ and $Z_{k,n+1}$. This recurrence is not visible at the level of stable envelopes or their integrals, even for $k=1$, suggesting a hidden symmetry to the geometric setup. 

To a $k$-element subset of $[n]$, $I=\{i_1<i_2<\dots<i_k\}$, consider the corresponding collection of $k$-tuples of $[n]$ given by all possible ways to decrease any number of the indices of $I$ by one, i.e the set $S_{I}$ defined by
\[
S_{I}:=\big\{(i_{1}-\alpha_{1},\dots,i_{k}-\alpha_{k}),\ \alpha_{j}\in\{0,1\}\big\}
\]

Given a fixed point $p_I \in X_{k,n}^\bT$, one is tempted to think of the set $S_I$ as all the ways one can naturally associate a fixed point $p_J\in X_{k,n-1 }^\bT$. The association $p_I\to I \to J \to p_J$ can fail in three ways:
\begin{enumerate}\label{eq:cases}
    \item One of the indices of $J$ is 0, i.e. if $1\in I$.
    \item $J$ has a repeated index, i.e. if sequence $i,i+1 \in I$ for some $i$. \hfill$(\star)$
    \item $n\in I$, and the element $n$ remains in $J$. 
\end{enumerate}
These exceptions are crucial for the statement of the recurrence relation. In particular, though all 3 cases do not yield genuine fixed points of $X_{k,n-1}$, one can still evaluate the expression (\ref{closed-form}) at the resulting $k$-tuple. It is clear that (\ref{closed-form}) vanishes for tuples arising from case 3), and we expect that it also vanishes for case 1). However tuples from case 2) can yield nonzero integers.
Our main conjecture asserts a recurrence relation between $Z_{k,n}$ and $Z_{k,n-1}$, generalizing Pascal's rule. We refer to this recurrence relation as $2^k$-neighbor addition, for reasons which will become apparent. 

\begin{conjecture}\label{2^k-nbr}
Let $p_I\in X_{k,n}^{\bT}$ with $I=\{i_{1}<\dots<i_{k}\}$. Then 
\[
\int_{X_{k,n}}\Stab(p_{I})=\sum_{J\in S_{I}}\int_{X_{k,n-1}}\Stab(p_{J})
\]
For $J$ falling in cases 1), 2), or 3) in $(\star)$, $\int_{X_{k,n-1}}\Stab(p_J)$ is undefined. For those cases, we declare that the summand is given by evaluating the expression (\ref{closed-form}) at the $k$-tuple $J$. 
\end{conjecture}

\begin{remark}
There is a ``naive version'' of Conjecture \ref{2^k-nbr} which only takes into account the geometric setting: one would say to leave any $J$ which do not correspond to a fixed point as contributing 0 to the sum, rather than evaluating $J$ at (\ref{closed-form}). For a given $k,n$, this naive version does hold for most fixed points, but is not true in general. In this way, we view the inclusion of the exceptional cases in 2) as ``corrections'' to the naive conjecture. 
\end{remark}

\begin{definition}
    A sequence $(a_{n})_{n\in\mathbb{Z}_{\geq 0}}$ of positive numbers is said to be log-concave if
    \[
    a_{i}^{2}\geq a_{i-1}a_{i+1}\ \quad \forall i\geq 1
    \]
\end{definition}
It is well-known that the sequence of binomial coefficients $\left(\binom{n}{i} \right )_{0\leq i\leq n}$ is log-concave. We also observe log-concavity of naturally defined subsequences within $Z_{k,n}$. 

\begin{conjecture}
    Let $\underline{d}:=(d_{1},\dots,d_{k-1})$ be a $k-1$ tuple of positive integers such that \newline$\sum_{i}d_{i}\leq n-1$, and let $\underline{d}_{i}$ denote the $k-$subset $\{i,i+d_{1},\dots, i+d_{k-1}\}$ Then the sequence 
    \[
    \left(\int_{X_{k,n}}\Stab(p_{\underline{d}_{i}})\right)_{1\leq i\leq n-d_{k-1}}
    \]
    is log-concave.
\end{conjecture}
\begin{example}\label{log-concavity-example}
    Let $n=10,\ k=3$, and consider $\underline{d}=(2,1)$. The sequence of fixed points is given by
    \[
    (\underline{d}_{i})_{1\leq i\leq n-\sum_{i}d_{i}}=\Big(\{1,3,4\},\ \{2,4,5\},\ \{3,5,6\},\ \{4,6,7\},\ \{5,7,8\},\ \{6,8,9\},\ \{7,9,10\}\Big)
    \] and the corresponding sequence of integrals is given by
    \[\left(\int_{X_{k,n-1}}\Stab(p_{\underline{d}_{i}})\right)_{1\leq i\leq n-\sum_{i}d_{i}}=\Big(30,\ 669,\ 3110,\ 4270,\ 1986,\ 295,\ 9\Big)\]
    Which is log-concave. See Remark \ref{remk-slice} for more on the sequence $(\underline{d}_{i})_{1\leq i\leq n-\sum_{i}d_{i}}$
\end{example}

For the rest of this section we set $k=2$, and explain the $2^k$-neighbor addition recurrence relation in more detail. In this setting it is easier to see why the recurrence relation has the name of ``4-neighbor addition'', and how it generalizes Pascal's rule, and we show how this recurrence relation can be used to calculate a generating function for $Z_{2,n}$.

\begin{example}[$X_{2,4}$]\label{24eg}

For a fixed $n$, we arrange the fixed points of $X_{2,k}$ in a triangle: We start by placing the highest fixed point in the lower left corner of the triangle, in this case $\langle 34\rangle$, where $\langle ij\rangle$ denotes the fixed point $\text{span}(e_i,e_j)$. Then we declare that traveling in the ``upper right'' direction sends $\langle ij\rangle\to\langle i-1,j\rangle$, and traveling in the ``right'' direction sends $\langle ij\rangle\to\langle i-1,j-1\rangle$. This results in the triangle
\[
\begin{tabular}{ccccccc}
&& $\langle 14\rangle$&& \\
&& &&\\
&$\langle 24\rangle $& &$\langle 13 \rangle$\\
&&&\\ 
$\langle 34\rangle$&&$\langle 23\rangle$&&$\langle 12\rangle$
\end{tabular}.
\]
The resulting triangles formed for particular choices of $n$ are referred to as layers. We then construct another triangle of the same size whose entry in location $\langle ij\rangle$ is the corresponding integer in $Z_{2,n}$:
$$
\begin{tabular}{ccccc}
&&2&&\\
&3&&3&\\
1&&2&&1
\end{tabular},\quad \quad\quad\begin{tabular}{ccccccc}
&& $\langle 14\rangle$&& \\
&& &&\\
&$\langle 24\rangle $& &$\langle 13 \rangle$\\
&&&\\ 
$\langle 34\rangle$&&$\langle 23\rangle$&&$\langle 12\rangle$
\end{tabular}
$$
For instance, 
\[
\int_{X_{2,4}}\hbar^4\stab(p_{\{2,3\}})=2
\]

\end{example}

This construction yields, for every $n\geq 2$, a combinatorial way of arranging the integers $Z_{2,n}$ in the plane.

The number of fixed points is such that if we stack these triangles on top of each other in 3-space, they will form a 3-simplex or pyramid. We call the resulting infinite 3-simplex the \underline{$Gr_2$-simplex}. It is naturally indexed by $n\geq 2$, whose horizontal cross-sections (layers) consist of the triangle formed by computing limits of integrals of stable envelopes for a fixed $n$. A list of the $Z_{2,n}$ integers can be found at OEIS A390353.

Note that the difference $j-i$ is constant along rows of each layer, and $j$ is constant along the north-east and south-west diagonals of each layer. The triples $(layer, row,\ $NE-SW$ \ diagonal)$ determines a coordinate system for the $Gr_2$-simplex. The 4-neighbor addition can be understood graphically in the following way.

\begin{example}\label{good4neighbor}
We place the integers $Z_{2,4}$ in a triangle (in blue) ``above" the integers $Z_{2,5}$, also in a triangle (in black). We further include integers arising from evaluating the $2$-tuples from $(\star)$ case 2), according to the same coordinate system (in the case $k=2$, this just means they are placed in the bottom-most row). In this case, there are 4 exceptions, $\langle 11\rangle, \dots, \langle 44\rangle$. 

Pictured below is a top-down view, with the blue layer closer to the reader. With these conventions, 4 neighbor addition manifests as follows: A black integer is determined by the 4 blue/red integers surrounding it at 12 o'clock, 3 o'clock, 6 o'clock, and 9 o'clock. 
$$
\begin{tabular}{cccccccccc}
&&&2&&\\
&&&\textcolor{blue}{2}\\
&&5&&5&\\
&&\textcolor{blue}{3}&&\textcolor{blue}{3}\\
&4&&10&&4\\
&\textcolor{blue}{1}&&\textcolor{blue}{2}&&\textcolor{blue}{1}\\
1&&4&&4&&1\\
\textcolor{red}{0}& & \textcolor{red}{-2} & & \textcolor{red}{-2} & & \textcolor{red}{0}
\end{tabular}
$$
It is not clear what geometric significance these nonpositive correction entries have.
\end{example}

Given any layer of the $Gr_2$-simplex, the integers in this layer can be naturally assigned as the coefficients of a complete homogeneous polynomial in 3 variables, say $a,b,c$, and the resulting polynomial will be divisible by $a+2b+c$. This is illustrated for the first few layers in Figure \ref{fig: unreduced gr2 example}.

\begin{figure}[h!]
    \centering
    \includegraphics[width=0.75\linewidth]{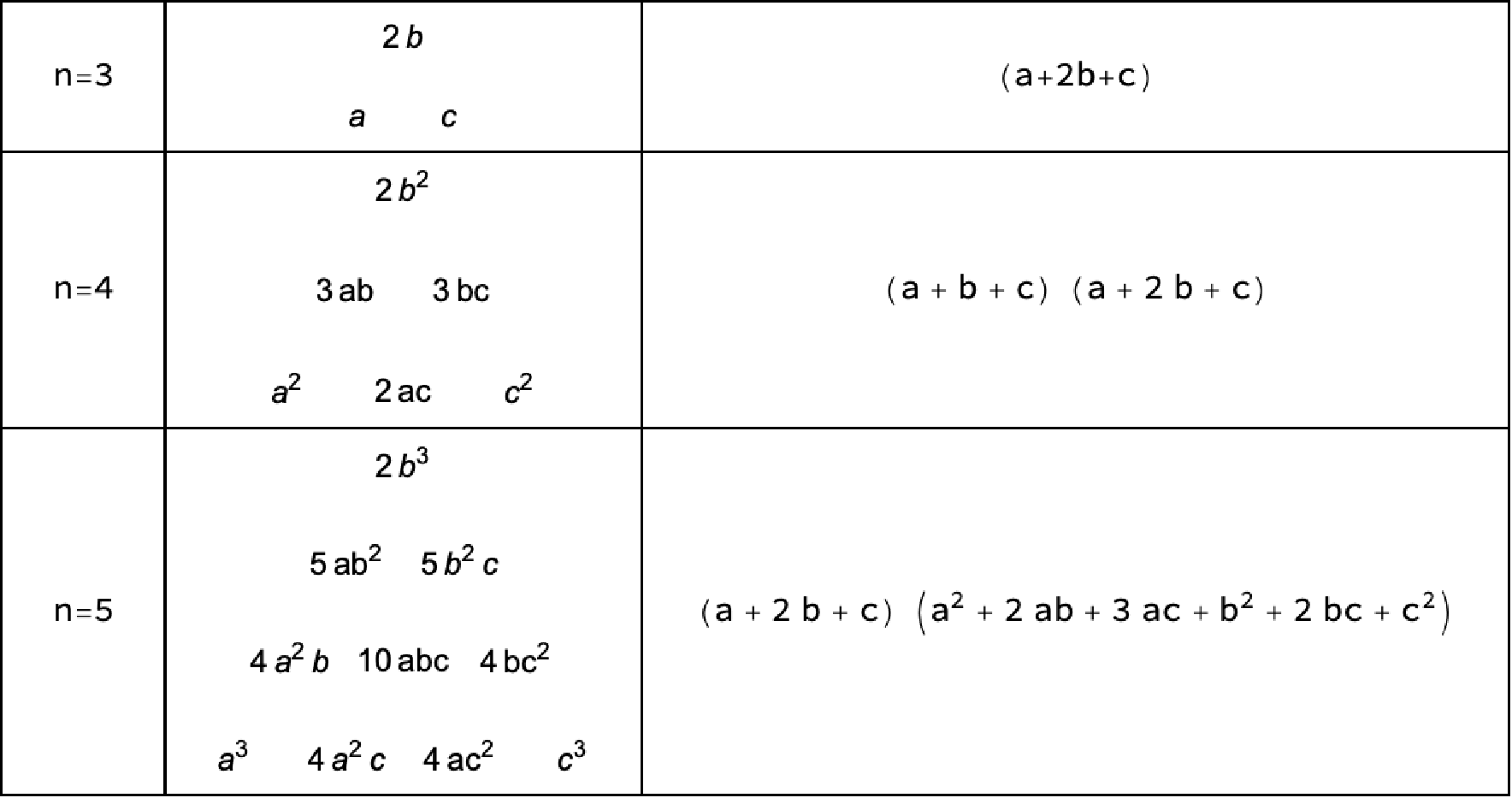}
    \captionsetup{width=.8\linewidth}
    \caption{Integers from $Z_{2,n}$ for $n=3,4,5$ arranged into a triangle, then associated to the coefficients of a polynomial in $a,b,c$. The resulting polynomial is divisible by $(a+2b+c)$. }
    \label{fig: unreduced gr2 example}
\end{figure}

The monomials are arranged as follows: The integer at the top of the triangle corresponds to $b^{n-2}$, and going southeast in the triangle corresponds to multiplying the current monomial by $\frac{c}{b}$, and going southwest in the triangle corresponds to multiplying the current monomial by $\frac{a}{b}$.

\begin{conjecture}
\label{4-nbr-Gr2}
    The entries in each layer of the $Gr_{2}$-simplex correspond (via the arrangement described above) to coefficients of a polynomial in $a,\ b,\ c$ that is divisible by $(a+2b+c)$. 
\end{conjecture}

If we divide the polynomial of a given layer by the common factor $(a+2b+c)$, we obtain a polynomial of one degree less, whose coefficients can be naturally arranged into a triangle one size smaller, in the same manner as described above. Then we can form a new 3-simplex by assembling these ``reduced'' triangles, which will now start at a single point at the top. We call this simplex the \underline{reduced $Gr_2$-simplex}. The reduced $Gr_2$-simplex does not have natural coordinates given by fixed point labeling, but does inherit the same coordinates given by layer, row, $NE-SW$-diagonal. We show a few layers of this process in Figure \ref{fig: reduced gr2}. To distinguish the coordinates on the reduced $Gr_2$-simplex from the coordinates on the (unreduced) $Gr_2$-simplex, we denote layers of the former by $\ell$, and the latter by $n$. 

\begin{figure}[h!]
    \centering
    \includegraphics[width=0.75\linewidth]{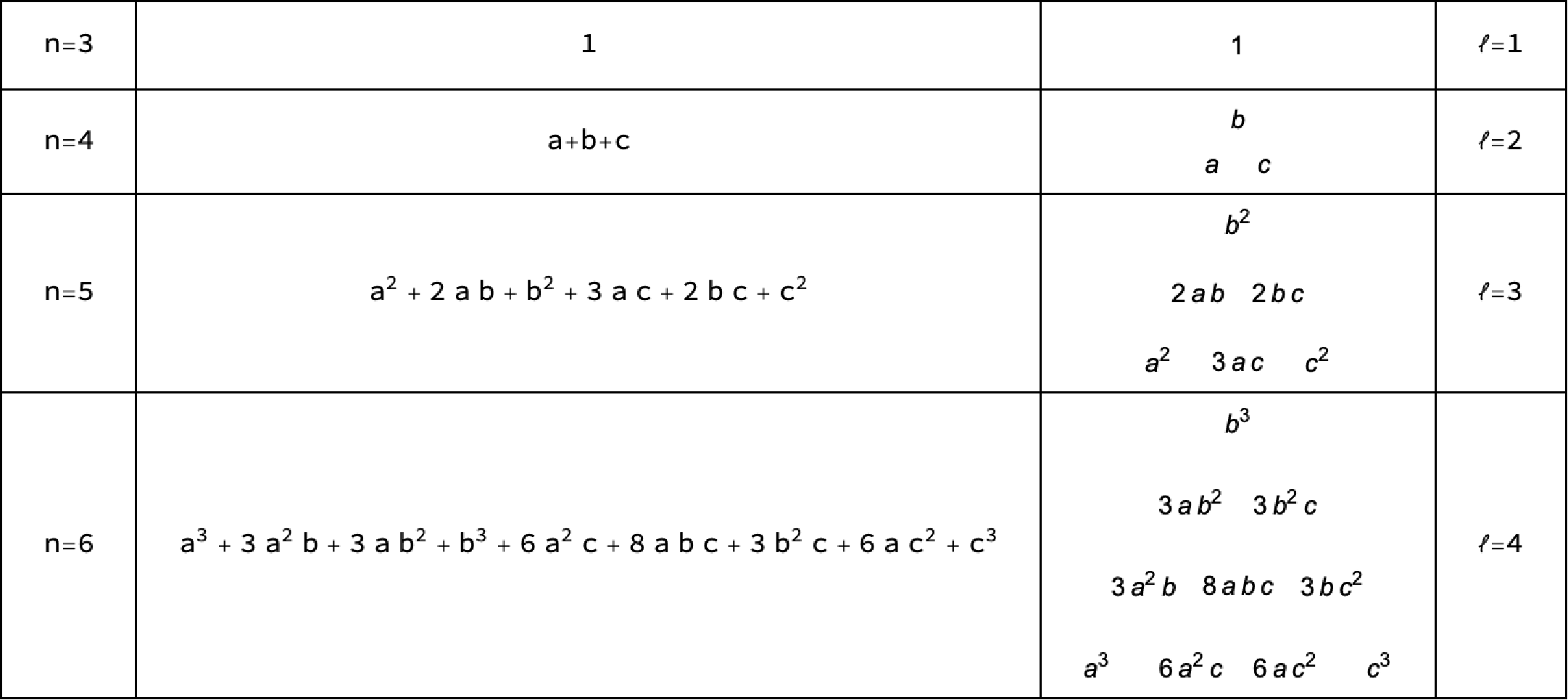}
    \captionsetup{width=.8\linewidth}
    \caption{From the $Gr_2$-simplex, we divide out the common factor of $(a+2b+c)$. The result can be arranged into another 3-simplex, which we call the reduced $Gr_2$-simplex.}
    \label{fig: reduced gr2}
\end{figure}

The reduced $Gr_2$-simplex also obeys a 4-neighbor addition recurrence relation, but does not require the additional correction terms, as in the (un-reduced) $Gr_2$-simplex, see Example (\ref{good4neighbor}).

\begin{conjecture}\label{reduced-Gr2-4-nbr}
\[\xi_{\ell,r,c}=\xi_{\ell-1,r,c}+\xi_{\ell-1,r-1,c}+\xi_{\ell-1,r-1,c-1}+\xi_{\ell-1,r-2,c-1}\]
Where $\xi_{\ell,r,c}$ denotes the integer in the reduced $Gr_{2}$-simplex in layer $\ell$, row $r$, and NE-SW diagonal $c$, and $\xi_{\ell,r,c}:=0$ if there is no entry in this position in the reduced $Gr_{2}$-simplex.
\end{conjecture}
For example, below in blue is layer $\ell=3$ of the reduced $Gr_2$-simplex which should be thought of as sitting above layer $\ell=4$ of the reduced $Gr_2$-simplex in black. Layer $\ell$ of the reduced $Gr_2$-simplex corresponds to layer $\ell+1$ of the $Gr_2$-simplex.

$$
\begin{tabular}{cccccccccc}
&&&1&&\\
&&&\textcolor{blue}{1}\\
&&3&&3&\\
&&\textcolor{blue}{2}&&\textcolor{blue}{2}\\
&3&&8&&3\\
&\textcolor{blue}{1}&&\textcolor{blue}{3}&&\textcolor{blue}{1}\\
1&&6&&6&&1
\end{tabular}
$$

The reduced $Gr_2$-simplex and the $Gr_2$-simplex contain equivalent information, and we can translate between them by multiplying or dividing by $(a+2b+c)$. 

The following is a computation of the generating function for the reduced $Gr_2$-simplex.

Let $\xi_{\ell,r,c}$ be defined as in Conjecture \ref{reduced-Gr2-4-nbr}. Since the reduced $Gr_{2}$-simplex satisfies 4-neighbor addition without corrections, each entry satisfies the recursion in \ref{reduced-Gr2-4-nbr}.
Multiplying through by $x^\ell y^r z^c$ on both sides of this recursion, and summing over the valid ranges for each index, and defining $B(x,y,z) := \sum\limits_{1 \leq c \leq r \leq \ell} \xi_{\ell,r,c} x^\ell y^r z^c$ to be the ordinary generating function for the reduced $Gr_{2}$-simplex, we may derive the expression for $B(x,y,z)$ following classical generating function techniques as in \cite{W}:
\begin{align*}
B(x,y,z)-xyz
&= B(x,y,z)(x+xy+xyz)+xy^2z \left( B(x,y,z)-\sum_{\substack{\ell \geq 1\\ 1 \leq c \leq \ell}} \xi_{\ell,\ell,c}(xy)^{\ell}z^{c}\right)\\
\implies B(x,y,z) &= \frac{xyz - \frac{yz}{2}\left(1-xy(1+z)-\sqrt{(1-xy(1+z))^2-4(xy)^2z}\right)}{1-x(1+y)(1+yz)},
\end{align*}
where we note that the sum above is the generating function for the Narayana numbers \cite{OEISNarayana}.

We can work interchangeably with the variables $x,\ y,\ z$ and the variables $a,\ b,\ c$ by the maps:
\begin{align*}
x^\ell y^r z^k &\rightarrow a^{r-k}b^{\ell-r}c^{k-1}\\
x^{i+j+k+1}y^{i+k+1}z^{k+1} &\leftarrow a^ib^jc^k
\end{align*}
This means we can convert the reduced $Gr_{2}$-simplex to the $Gr_{2}$-simplex by multiplying by $(xy+2x+xyz)$ (This is the monomial that corresponds to $a+2b+c$ after the change of variables to $x,\ y,\ z$), and then adding $xyz$, which corresponds to $X_{2,2}$. Then the generating function for the $Gr_{2}$-simplex is:
\begin{align*}
F(x,y,z) &= xyz + (xy+2x+xyz)B(x,y,z)\\
&= x y z \left(1-\frac{y (2 x z+1) \left(x (y z+y+2)+\sqrt{x y \left(x y (z-1)^2-2 (z+1)\right)+1}-1\right)}{2 x (y+1) (y z+1)-2}\right)
\end{align*}

To compute the entry corresponding to the fixed point $\langle i_1,i_2 \rangle$ for a specific $n \geq 2$, one need only find the coefficient of $x^{n-1}y^{n-(i_2-i_1)}z^{i_1}$ in $F(x,y,z)$. Taking this coefficient must equivalent to evaluating the expression in Theorem \ref{lim-calc} with $k = 2$, and $I = \{i_1,i_2\}$, though it is not obvious.

For higher values of $k$, we expect there should be an analogous Pascal's $n$-simplex obeying $2^k$-neighbor addition. 
\begin{remark}\label{remk-slice}
The sequence of fixed points $(\underline{d}_{i})_{1\leq i\leq n-d_{k-1}}$ in Example \ref{log-concavity-example} comes from a ``1-dimensional slice" of the $n$-simplex obtained from our fixed points.
\end{remark}
\subsection{Paths to Young diagrams}\label{var-path}

This section outlines a new conjectural procedure to obtain integrals of stable envelopes $\stab_(p_{I})$ due to A. Varchenko (for the standard chamber). 

We denote the multiset consisting of $k_{i}$ many copies of $\alpha_{i}$ by $\sum_{i}k_{i}\alpha_{i}$. Let $\lambda$ be a partition whose Young diagram fits in an $(n-k)\times k$ box. We can associate to this partition a subset of $\{1,\dots,n\}$ of size $k$ as follows: Consider the conjugate partition $\lambda^{t}$ of $\lambda$ and append 0 if necessary so that $\lambda^{t}$ is a $k-$tuple. The subset associated to $\lambda$ is the set $\Omega(\lambda):=\lambda^{t}+(k,k-1,\dots,1)$. We describe a procedure that defines a rational function in the equivariant parameters $a_{i}$ for a given partition $\lambda$:\newline\newline
Consider a partition $\lambda$ in an $(n-k)\times k$ rectangle, and let $\lambda^{c}$ denote its complement in this rectangle. Call a sequence $\mathcal{S}$ of 2-tuples a \underline{path} to $\lambda$ if it contains $(n-k)\times k$ entries so that each entry is the coordinates of a box in the $(n-k)\times k$ rectangle (the first entry is the row number, and the second is the column), and the subsequence $\mathcal{S}_{\lambda}$ of coordinates consisting of boxes in $\lambda$ is a valid way of removing boxes from $\lambda$ to get to the empty partition, similarly the subsequence $\mathcal{S}_{\lambda^{c}}$ is a valid way of removing boxes from $\lambda^{c}$ to get to the empty partition, where $\lambda^{c}\cup\lambda$ is the $(n-k)\times k$ rectangle. Let $\mathcal{P}_{\lambda}$ denote the set of paths to $\lambda$. Given a path $\mathcal{S}$ to $\lambda$, define a sequence $\mathcal{T}(\mathcal{S})$ of multisets as follows: The $j$-th multiset $\mathcal{T}(\mathcal{S})_{j}$ in this sequence is obtained by $\sum\limits_{i=1}^{j}a_{e_{i}(1)-e_{i}(2)+k}$, where $\mathcal{S}=((e_{i}(1),e_{i}(2)))_{i=1}^{k(n-k)}$. Given a multiset $A=\sum_{i}k_{i}\alpha_{i}$, define
\[w(A):=\sum_{i}\hbar k_{i}(k_{i}-k_{i+1})+k_{i}(a_{i+1}-a_{i})\]
Consider the sum 
\[\mathbb{V}(\lambda):=\sum\limits_{\mathcal{S}\in\mathcal{P}_{\lambda}}\prod\limits_{j=1}^{k(n-k)}\frac{1}{w(\mathcal{T}(\mathcal{S})_{j})}\]

\begin{conjecture}\label{var-int}\textbf{(A. Varchenko)}
Given a partition $\lambda$ in an $(n-k)\times k$ rectangle,\[\mathbb{V}(\lambda)=\sum_{p_J\in X_{k,n}^\bT}\frac{\Stab(p_{\Omega(\lambda)})|_{p_{J}}}{e(T_{p_J}X_{k,n})}\]
   This implies that 
   \[\lim_{\mathbf{a}\to 0}\mathbb{V}(\lambda)=\int_{X_{k,n}}\Stab(p_{\Omega(\lambda)})\]
\end{conjecture}
\begin{example}
    Consider the case $X_{2,4}$, with partition $\lambda:=(1)$ in the $2\times 2$ rectangle. Then $\Omega(\lambda)=\{1,3\}$. 
    
    The complement of $\lambda$ in the $2\times 2$ rectangle is $(2,1)$. $\lambda$ consists of the boxes $(1,1)$, and $\lambda^{c}$ consists of the boxes $(1,2),\ (2,1), (2,2)$. The paths as defined in the section above are given by:
    \begin{align*}
    [(1, 1), (2, 1), (1, 2), (2, 2)]\\
    [(1, 1), (1, 2), (2, 1), (2, 2)]\\ 
    [(2, 1), (1, 1), (1, 2), (2, 2)]\\ 
    [(2, 1), (1, 2), (1, 1), (2, 2)]\\
    [(2, 1), (1, 2), (2, 2), (1, 1)]\\
    [(1, 2), (1, 1), (2, 1), (2, 2)]\\
    [(1, 2), (2, 1), (1, 1), (2, 2)]\\
    [(1, 2), (2, 1), (2, 2), (1, 1)]
    \end{align*}
    To simplify our presentation, we calculate the summand of $\mathbb{V}(\lambda)$ for the first path above $[(1, 1), (2, 1), (1, 2), (2, 2)]$.\\
    Pictorially, the path to $\lambda=(1)$ that we consider above corresponds to the following 2 sequences of Young diagrams: The only box belonging to $\lambda$ is $(1,1)$, and the rest of the boxes belong to $\lambda^{c}$. The subsequence of boxes in $\lambda$ corresponds to the following way of removing boxes from $\lambda$:
    \[\ydiagram{1}\ \mathlarger{\mathlarger{\mathlarger{\supset}}}\ \mathlarger{\mathlarger{\mathlarger{\varnothing}}}\]
    Similarly the boxes of $\lambda^{c}$ correspond to removing boxes of $\lambda^{c}$ by:
    \[\ydiagram{2,1}\ \mathlarger{\mathlarger{\mathlarger{\supset}}}\ \ydiagram{1,1}\ \mathlarger{\mathlarger{\mathlarger{\supset}}}\ \ydiagram{1}\ \mathlarger{\mathlarger{\mathlarger{\supset}}}\ \mathlarger{\mathlarger{\mathlarger{\varnothing}}}\]
    The sequence of multisets for this path is given by
    \[\mathcal{T}(\mathcal{S}):=[a_{2},\ a_{2} + a_{3},\ a_{2} + a_{3} + a_{1},\ 2a_{2} + a_{3} + a_{1}]\]
    And finally
    \[\prod\limits_{j=1}^{4}\frac{1}{w(\mathcal{T}(\mathcal{S})_{j})}=\frac{1}{(\hbar + a_{3} - a_{2})(\hbar - a_{2} + a_{4})(\hbar - a_{1} + a_{4})(2\hbar - a_{2} - a_{1} + a_{3} + a_{4})}\]
Carrying out the same process for all other paths yields
    \begin{align*}\mathbb{V}(\lambda)&=\sum\limits_{\mathcal{S}\in\mathcal{P}_{\lambda}}\prod\limits_{j=1}^{k(n-k)}\frac{1}{w(\mathcal{T}(\mathcal{S})_{j})}\\
    &=\frac{3\hbar^{2}-2\hbar a_{1}-\hbar a_{2}+a_{1}a_{2}+\hbar a_{3}-a_{2}a_{3}+2\hbar a_{4}-a_{1}a_{4}+a_{3}a_{4}}{(\hbar+a_{2}-a_{1})(\hbar+a_{3}-a_{1})(\hbar+a_{4}-a_{1})(\hbar+a_{3}-a_{2})(\hbar+a_{4}-a_{2})(\hbar+a_{4}-a_{3})}
    \end{align*}
    which matches matches the expression obtained by localization over $\bT$. 
\end{example}

\begin{conjecture}\label{var-complement}
    \[
    \mathbb{V}(\lambda)=\alpha(\mathbb{V}(\lambda^{c}))
    \]
    Where $\alpha$ is the automorphism of $\mathbb{C}(\hbar,a_{1},\dots,a_{n})$ induced by $a_{i}\mapsto -a_{n+1-i},\ 1\leq i\leq n$, leaving $\hbar$ fixed.
\end{conjecture}

This conjecture is almost an analogy of the chamber dependence result of Proposition (\ref{chamber-dependence}), when $\tau=(n,n-1,\dots,1)$, i.e. the opposite chamber to the standard, except the equivariant parameters are all negated. It is not clear if there is a way to incorporate chamber dependence for other chambers.

For a given partition $\lambda$ in an $(n-k)\times k$ rectangle, there is a natural bijection between paths to $\lambda$ and $\lambda^{c}$ given by sending the box $(e_{1},e_{2})$ to $(n-k+1-e_{1},k+1-e_{2})$. In spite of this bijection, Conjecture \ref{var-complement} does not hold at the level of multisets, i.e applying the map $\alpha$ to each element of the multisets corresponding to $\lambda$ does not give the set of multisets corresponding to $\lambda^{c}$.

\section{Generalizations in Type A}\label{sec5}

For this section we discuss only type A bow and quiver varieties. The main content of this section is to conjecture about how far this phenomenon can be carried, and will contain no proofs. We find that it does not extend to all bow varieties, but we conjecture that it does extend to all quiver varieties. For this section we assume all stability conditions are generic. 

\subsection{Bow variety limits do not exist}\label{bows}
Because we don't offer any proofs at this level, we refrain from providing a full background and definition, and instead refer the reader to a source for stable envelopes and weight functions on bow varieties, \cite{rim-shou}. In this section, $X$ denotes an arbitrary bow variety.

Briefly, (type A) bow varieties are holomorphic symplectic varieties generalizing type A quiver varieties, in the sense that quiver varieties can be identified with a proper subset of bow varieties. Bow varieties are assembled from combinatorial data called brane diagrams in a manner similar to the assembly of quiver varieties from quiver diagrams. They admit a torus action by the torus $\bT=(\C^*)^{|\text{D5-branes}|}\times\C^*_\hbar\equiv \bA\times\C^*_\hbar$. In the source, they develop the theory of cohomological stable envelopes for an appropriate $\bT$-action on these bow varieties. 

Our definition of integral of stable envelope as localization over $\C^*_\hbar$ no longer holds, as there is no guarantee that $X^{\C^*_\hbar}$ is compact. However, we can still compute the localization over $\bT$, as the fixed point set is finite, and take the non-equivariant limits. For this section only, we abuse the notation and write 

\[
\lim_{\textbf{a}\to 0}\int_X\Stab(p)
\]

where the integral is defined via localization over $\bT$ (the same formula as always). 

We are naturally led to ask if the limits $\textbf{a}\to 0$ exist for all bow varieties. This example shows that these limits can fail to exist. 

\begin{example}
Consider the bow variety, $X:=\mathcal C(\mathcal D)$, associated to the brane diagram
\begin{gather*}
\mathcal D= \Df 2 \NSf 2 \Df 2 \NSf 1 \NSf 1 \Df
\end{gather*}
It has charge vectors $r=(1,1,2)$ and $c=(1,2,1)$. As the vector of NS5 charges (the $r$ vector) is not weakly decreasing, this brane diagram is not Hanany-Witten equivalent to a co-balanced brane diagram, thus this bow variety is not Hanany-Witten equivalent to a quiver variety. These facts are theorems of the source.

For the torus action, the set of fixed points on a bow variety have many combinatorial encodings, discussed in Section 4 of the source. In terms of binary contingency tables (BCTs), the fixed point $p$ given by
    \begin{table}[h!]
        \centering
        \begin{tabular}{c|c|c|c|}
             &$U_1$ & $U_2$ & $U_3$\\
             \hline
             $V_1$&0&1  &0 \\\hline
             $V_2$&0& 1 &0 \\\hline
             $V_3$&1&0  &1 \\
             \hline
        \end{tabular}
    \end{table}
    \\
has integral (after fully simplifying)
\[
\int_{X}\Stab(p)=\frac{-3\hbar-a_3+a_1}{(a_3-a_2+2\hbar)(a_1-a_2)(a_3-a_1+\hbar)(\hbar+a_2-a_1)}
\]
\end{example}

As mentioned before, we do not have a conceptual reason why these limits can fail to exist, in contrast to the conjecture in the following section, but finding one would be of interest. 

\subsection{Type A quiver variety limits do exist}

Nakajima quiver varieties can be (identified with) special cases of bow varieties, for background on these spaces see Section 5 of the previous source, or \cite{NQV}. The introductory paragraph of \S\ref{bows} applies to quiver varieties as well, but now we conjecture the limits do exist 

\begin{conjecture}
When $X$ is a type A Nakajima quiver variety, and $p$ is a $\bT$-fixed point for the natural torus action,

\[
\hbar^{\dim(X)/2}\lim_{\mathbf{a}\to 0}\int_X\Stab(p) \quad \quad \mathrm{exists}
\]
\end{conjecture}

This means that the bow varieties which arise from quiver varieties also have all nonequivariant limits of integrals of stable envelopes. An interesting property we observe is that, among the bow varieties, this existence appears to be equivalent to the condition that the bow variety is Hanany-Witten equivalent to a quiver variety: 

\begin{conjecture}\label{bow}
A bow variety $X$ is Hanany-Witten equivalent to a quiver variety if and only if the non-equivariant limit of the integral of the stable envelope $\displaystyle \lim_{\textbf{a}\to 0}\int_{X}\Stab(p)$ exists for all torus fixed points $p$. 
\end{conjecture}

\appendix
\section{Existence of nonequivariant limit}
Though this theorem was proven indirectly, here we provide a direct, combinatorial proof.

\begin{theorem}\label{lim-exists} The following limit exists for all $p_I\in X_{k,n}^\bT$:
\[\lim_{\mathbf{a}\to 0} \left( \sum_{p_J\in X_{k,n}^\bT}\frac{\Stab(p)}{e(T_{p_J}X_{k,n})}\right) \]
\end{theorem}
\begin{proof}~\\
Consider a fixed point $I=\{i_{1}<\dots<i_{k}\}$. The weight function is a rational function in $t_{1},\dots, t_{k}$ which is symmetric in $t_{1},\dots,t_{k}$ by definition, and the weight function of $p_{I}$ is given by
\[W(p_I)=\sum\limits_{\sigma \in S_k}\dfrac{\prod\limits_{r = 1}^k \prod\limits_{\alpha = 1}^{i_{r} - 1} (a_{\alpha} - t_{\sigma(r)}) \prod\limits_{\beta = i_r + 1}^n (t_{\sigma(r)} - a_{\beta} + \hbar) }{\prod\limits_{1 \leq \ell < m \leq k} (t_{\sigma(\ell)} - t_{\sigma(m)}) (t_{\sigma (\ell)} - t_{\sigma(m)} + \hbar)}\]
Note that 
\[
\sgn(\sigma)\cdot\prod\limits_{1\leq \ell<m\leq k}(t_{\ell}-t_{m})=\prod\limits_{1\leq \ell<m\leq k}(t_{\sigma(\ell)}-t_{\sigma(m)})\]
and
\[
\prod\limits_{1\leq l<m\leq k}(t_{\sigma(\ell)}-t_{\sigma(m)}+\hbar)(t_{\sigma(m)}-t_{\sigma(\ell)}+\hbar)=\prod\limits_{1\leq l<m\leq k}(t_{\ell}-t_{m}+\hbar)(t_{m}-t_{\ell}+\hbar)\] 
for any $\sigma\in S_{k}$.\\
Define \[
D:=\prod\limits_{1\leq\ell<m\leq k}(t_{\ell}-t_{m})(t_{\ell}-t_{m}+\hbar)(t_{m}-t_{\ell}+\hbar)
\]
The weight function associated to a fixed point may then be written as
\[
 W(p_{I})=\frac{N}{D}
\]
where $N,\ D$ are polynomials in the $a_{i}$, $t_{i}$ and $\hbar$. Since $W_{p_{I}}$ is a rational function in $t_{1},\dots, t_{k}$ and $D$ can also be thought of as a polynomial that is skew-symmetric in $t_{1},\dots,t_{k}$, we can conclude that numerator $N$ must also be skew symmetric in $t_{1},\dots,t_{k}$, which implies that the expression $V_{t}:=\prod\limits_{1\leq\ell<m\leq k}(t_{\ell}-t_{m})$ divides $N$. On specializing $t_{s}=a_{j_{s}}$ for any fixed point $J=\{j_{1}< \cdots < j_{k}\}$, we see that $N'_{J}:=\frac{N}{V_{t}}|_{t_{s}=a_{j_{s}}}$ and $D'_{J}:=\frac{D}{V_{t}}|_{t_{s}=a_{j_{s}}}$ are polynomials in the $a_{i}$ and $\hbar$. \\
For any ordered subset $X_{k,n}=\{x_{1},\dots,x_{p}\}$ define
\[V_{X_{k,n}}:=\prod\limits_{1\leq i<j\leq p}(x_{i}-x_{j})\]
For any fixed point $J=\{j_{1},\dots,j_{k}\}$, and permutation $\sigma \in S_k$, we define the following expressions:
\begin{multicols}{2}
\begin{itemize}
\item $V_{t}^{\hbar} := \prod\limits_{1 \leq \ell < m \leq k} (t_{\ell} - t_{m} + \hbar)$,
\item $\overline{V_{t}^{\hbar}} := \prod\limits_{1 \leq \ell < m \leq k} (t_{m} - t_{\ell} + \hbar)$,
\item $\overline{V_{J,\sigma}^{\hbar}} := \prod\limits_{1 \leq \ell < m \leq k} (a_{j_{\sigma(m)}} - a_{j_{\sigma(\ell)}} + \hbar)$,
\item $V_J := \prod\limits_{1 \leq \ell < m \leq k} (a_{j_\ell} - a_{j_m})$,
\item $V_J^{\hbar} :=  \prod\limits_{1 \leq \ell < m \leq k} (a_{j_\ell} - a_{j_m} + \hbar)$,
\item $\overline{V_J^{\hbar}} := \prod\limits_{1 \leq \ell < m \leq k} (a_{j_m} - a_{j_\ell} + \hbar)$,
\item $V_{J,J^c}^{\hbar} := \prod\limits_{v \not\in J} \prod\limits_{p = 1}^{k} (a_{j_p} - a_{v} + \hbar)$,
\item $D_t := V_t V_t^{\hbar} \overline{V_{t}^{\hbar}}$,
\item $D_J := V_J V_J^{\hbar} \overline{V_{J}^{\hbar}}$,
\item $D_J' := \frac{D_J}{V_J} = V_J^{\hbar}\overline{V_J^{\hbar}}$,
\end{itemize}
\end{multicols}
We also define \[N_{J,t} := \sum\limits_{\sigma \in S_k} \sgn(\sigma) \left(\prod\limits_{r = 1}^k \prod\limits_{\alpha = 1}^{i_{r} - 1} (a_{\alpha} - t_{\sigma(r)}) \prod\limits_{\beta = i_r + 1}^n (t_{\sigma(r)} - a_{\beta} + \hbar) \right) \overline{V_{J,\sigma}^{\hbar}}.\]
With these definitions, the restriction $W(p_I)|_{p_J}$ is
\[W(p_I)|_{p_J} = \left.\dfrac{N_{J,t}}{V_{t}V_{t}^{\hbar}\overline{V_{t}^{\hbar}}}\right|_{t_s = a_{j_s}}.\]
Note that the denominator before restriction is $D_t$.
Letting $N_{J}:=N_{J,t}|_{t_{s}=a_{j_{s}}}$, we can again write $W(p_I)|_{p_J}$ as \[W(p_I)|_{p_J} = \dfrac{N_J}{D_J}.\](Observe that $D_{t}$ is skew symmetric in the $t_{s}$'s (and specializes to $D_{J}$ when $t_{s}=a_{j_{s}}$). Since the weight function is symmetric, and the denominator is skew symmetric, the numerator $N_{J,t}$ must be skew symmetric as well, which implies that the Vandermonde determinant 
\[\prod\limits_{1 \leq \ell < m \leq k}(t_{\ell}-t_{m})\]
divides $N_{J,t}$, in particular the quotient is a polynomial in the $t_{s}'s$, with coefficients that are polynomials in terms of the $a_{s}'s$ and $\hbar$.\\\\
This implies that when we specialize to $t_{s}=a_{j_{s}}$, we obtain that
\[N'_{J}:=\frac{N_{J}}{V_{J}}\]
\[D'_{J} =\frac{D_{J}}{V_{J}}\]
are polynomials in the $a_{s}$'s and $\hbar$.\\\\

Hence we can write the integral as
\begin{align*}
\int_{X_{k,n}}\Stab(p_{I})&=\sum_{p_{J}\in X_{k,n}^{T}}\frac{\frac{N_{J}}{D_{J}}}{\prod\limits_{v\not \in J} \prod\limits_{p = 1}^{k}(a_v - a_{j_p}) (a_{j_p}-a_v+\hbar)}\\
&=\sum_{p_J\in X_{k,n}^T}\dfrac{\sum\limits_{\sigma \in S_k} \dfrac{\prod\limits_{r = 1}^k \prod\limits_{\alpha = 1}^{i_{r} - 1} \left(a_{\alpha} - a_{j_{\sigma(r)}}\right) \prod\limits_{\beta = i_r + 1}^n \left(a_{j_{\sigma(r)}} - a_{\beta} + \hbar\right) }{\prod\limits_{1 \leq \ell < m \leq k} \left(a_{j_{\sigma(\ell)}} - a_{j_{\sigma(m)}}\right) \left(a_{j_{\sigma(\ell)}} - a_{j_{\sigma(m)}} + \hbar\right)}}{\prod\limits_{v\not \in J} \prod\limits_{p = 1}^{k}\left(a_v - a_{j_p}\right) \left(a_{j_p}-a_v+\hbar\right)}
\end{align*}
\begin{align*}
&=\sum_{p_J\in X_{k,n}^T}\sum\limits_{\sigma \in S_k} \frac{\prod\limits_{r = 1}^k \prod\limits_{\alpha = 1}^{i_{r} - 1} \left(a_{\alpha} - a_{j_{\sigma(r)}}\right) \prod\limits_{\beta = i_r + 1}^n \left(a_{j_{\sigma(r)}} - a_{\beta} + \hbar\right) }{\prod\limits_{1 \leq \ell < m \leq k} \left(a_{j_{\sigma(\ell)}} - a_{j_{\sigma(m)}}\right) \left(a_{j_{\sigma(\ell)}} - a_{j_{\sigma(m)}} + \hbar\right)\prod\limits_{v\not \in J} \prod\limits_{p = 1}^{k}\left(a_v - a_{j_p}\right) \left(a_{j_p}-a_v+\hbar\right)}\\
&=\frac{1}{V}\sum_{p_{J}\in X_{k,n}^{T}}\frac{N_{J}V_{J^{c}}(-1)^{kn+\binom{k}{2}+|J|}}{D'_{J}V_{J,J^{c}}^{\hbar}}\\
&=\frac{1}{V}\sum_{p_{J}\in X_{k,n}^{T}}\frac{N'_{J}V_{J}V_{J^{c}}(-1)^{kn+\binom{k}{2}+|J|}}{D'_{J}V_{J,J^{c}}^{\hbar}},
\end{align*}
where $V_{J,J^c}^{\hbar}:= \prod_{v \not \in J}\prod_{p = 1}^{k} \left(a_{j_p} - a_v + \hbar\right)$. The factor of $(-1)^{kn+\binom{k}{2}+|J|}$ appears from rearranging the terms of $\prod\limits_{v\notin J}\prod\limits_{p=1}^{k}\left(a_{v}-a_{j_{p}}\right)$ so that $v<j_{p}$ for each term in the product.\newline
To show that the required limit exists, it is enough to show that the expression 
\begin{equation}\label{eq:thesum}\sum_{p_{J}\in X_{k,n}^{T}}\frac{N'_{J}V_{J}V_{J^{c}}(-1)^{k(n-k)+\binom{k}{2}+|J|}}{D'_{J}V_{J,J^{c}}^{\hbar}}\end{equation}
vanishes whenever we set $a_{x}=a_{y}=z$ for any pair $1\leq x<y\leq n$, since this would imply that $V$ divides this sum, which in turn implies that the limit exists. We consider the following cases, and fix $1\leq x<y\leq n$:\newline

\subsection{$J$ for which $x,\ y\in J$ or $x,y \not\in J$} 
Observe that the summands in \eqref{eq:thesum} corresponding to fixed points $p_J$ with both $x$ and $y$ in $J$ vanish, since $(a_x - a_y)$ will appear in $V_J$. Similarly, the summands in \eqref{eq:thesum} corresponding to fixed points $p_J$ for which neither $x$ nor $y$ are in $J$ vanish, since $(a_{x}-a_{y})$ will appear in $V_{J^{c}}$. \\

Note that we are now left with 2 cases: one for which $x\in J,\ y\notin J$, and one for which $x\notin J,\ y\in J$.\\
\subsection{$J$ contains exactly one of $x$ and $y$}
 Consider the summands corresponding to fixed points $p_{J}$ for which $y \in J$ and $x \notin J$. We can disregard such $p_{J}$ that are not in the attracting set of $p_{I}$, since the restriction $W(p_I)|_{p_J}$ is 0.\\

We will first show that if $p_{J_{y}^{x}}$ (where $J^{x}_{y}:=\left(J\backslash\{y\}\right)\cup\{x\}$) is not in the attracting set of $p_{I}$, then $W(p_{I})|_{p_{J}}=0$

\subsubsection{$(J\backslash\{y\})\cup\{x\}\notin Attr^{f}(I)$}
Since $y \in J$, there exists an index $1 \leq r \leq k$ such that $y = j_r$. Since $x \not\in J$, there exists an index $1 \leq p \leq r$ such that $j_{p-1} < x < j_p$ (the edge case of $p = 1$ is handled similarly). Then we have that:
\begin{align*}
J &= \{j_1 < \ldots < j_{p-1} < j_p < j_{p+1} < \ldots < j_{r-1} < \overbrace{j_r}^{y} < j_{r+1} < \ldots < j_k\}\\
J_y^x &= \{j_1 < \ldots < j_{p-1} < x< j_p < \ldots < j_{r-2} < j_{r-1} < j_{r+1} < \ldots < j_k\}\\
I &= \{i_1 < \ldots < i_{p-1} < i_p < i_{p+1} < \ldots < i_{r-1} < i_r < i_{r+1} < \ldots < i_k\}
\end{align*}
Since $J_y^x \not\in Att(I)$, some element of $J_{y}^{x}$ is less than the element of $I$ in the same position. Specifically, since the first $p-1$ elements of $J_{y}^{x}$ are the same as the first $p-1$ elements of $J$, and the last $k-r$ elements of $J_{y}^{x}$ are the same the last $k-r$ elements of $J$, the elements of $J_{y}^{x}$ that are less than the corresponding elements of $I$ must be in positions that are at least $p$ and  at most $r$. Therefore, either $x < i_p$ or $j_{b} < i_{b+1}$ for some $p \leq b \leq r-1$.\\ We note that for any permutation $\sigma \in S_k$, the summand in $N_J$ corresponding to $\sigma$ will be 0 if $x$ is less than the element of $I$ in the position corresponding to the position that $y$ was sent to under the action of $\sigma$, because there will be a factor of $a_x - a_y$ from the $\alpha$ sum which is 0 upon setting $a_x = a_y = z$.\\
Let $p \leq q \leq r$ be the largest index such that the element of $J_y^x$ in position $q$ is less than the element of $I$ in position $q$. Then we have:
\begin{align*}
J &= \{j_1 < \ldots < j_{p-1} < j_p < j_{p+1} < \ldots < j_{q-1} < j_q < j_{q+1} < \ldots < j_{r-1} < \overbrace{j_r}^{y} < j_{r+1} < \ldots < j_k\}\\
J_y^x &= \{j_1 < \ldots < j_{p-1} < x< j_p < \ldots < j_{q-2} < j_{q-1} < j_{q} < \ldots < j_{r-2} < j_{r-1} < j_{r+1} < \ldots < j_k\}\\
I &= \{i_1 < \ldots < i_{p-1} < i_p < i_{p+1} < \ldots < i_{q-1} < i_q < i_{q+1} < \ldots < i_{r-1} < i_r < i_{r+1} < \ldots < i_k\}
\end{align*}
In $J$, for the elements in position $b$ for $b \leq q-1$, we have that the greatest position that $j_b$ can be sent to under the action of a given permutation $\sigma \in S_k$ is $q-1$ since $j_b \leq j_{q-1} < i_q$ for all $b \leq q-1$. Therefore, if $\sigma \in S_k$ is a permutation that sends any of the first $q-1$ elements of $J$ to a position that is at least $q$, then, without loss of generality, the element in position $l$ of $J$ is sent to position $u$ in $\sigma(J)$ where $l < q \leq u 
\leq k$. But we know in the definition of $N_J$, in the $\alpha$ product, there will be a factor of 0 if there exists an element of $J$ that is permuted to a position where it is less than the corresponding element of $I$. We see that the $u^{\text{th}}$ element of $\sigma(J)$, which is $j_l$, which satisfies $j_l < j_{q-1} < i_q \leq i_u$, hence the summand corresponding to $\sigma$ is 0. Hence we see that it must be the case that $\sigma$ permutes the first $q-1$ elements of $J$ among themselves and $\sigma$ permutes the last $k-q+1$ elements of $J$ among themselves.\\
Since $y = j_r \geq j_q$, the position of $y$, under the action of $\sigma$, is again at least position $q$, that is $j_{\sigma^{-1}(r)} \geq j_q > i_q > x$. Also, $x < i_q < i_{\sigma^{-1}(r)}$ since $q < \sigma^{-1}(r)$.\\
Thus $y$, which is the element in position $\sigma^{-1}(r)$ of $\sigma(J)$ where $\sigma^{-1}(r) \geq q$, is at least $j_q$ since $r \geq q$, and by definition $j_q \geq i_q$, and $i_q > x$ by definition of $q$. We also note that since the elements of $I$ are increasing, $i_{\sigma^{-1}(r)} > i_q$ as $\sigma^{-1}(r) > q$, so $x<i_{\sigma^{-1}(r)}$, so the summand corresponding to $\sigma$ becomes 0 when setting $a_x = a_y = z$.\\\\
Note that now we are left with 2 types of summands: One type corresponds to fixed points $p_{J}$ in the attracting set of $p_{I}$ that contain $x$ and not $y$, and another type that corresponds to fixed points $p_{J}$ in the attracting set of $p_{I}$ that contain $y$ and not $x$. Our final step in this proof will be to show that the summands corresponding to $p_{J}$ and $p_{\hat{J}}$ add up to 0 when we set $a_{x}=a_{y}=z$, where $p_{J}$ is a fixed point in the attracting set of $p_{I}$ containing $y$ and not $x$, and $p_{\hat{J}}$ is a fixed point in the attracting set of $p_{I}$ obtained from $p_{J}$ by replacing $y$ with $x$. \\

\subsubsection{$\hat{J}:=(J\backslash\{y\})\cup\{x\}\in Attr^{f}(I)$}
Let $r$ denote the position of the smallest entry in $J$ so that $j_{r}>x$, and let $s$ denote the position of $y$ in $J$. \\
Consider the sum 
\begin{equation}\label{twosummands}
    \frac{N_{\hat{J}}V_{\hat{J}^{c}}(-1)^{k(n-k)+\binom{k}{2}+|\hat{J}|}}{D'_{\hat{J}}V_{\hat{J},\hat{J}^{c}}}+
    \frac{N_{J}V_{J^{c}}(-1)^{k(n-k)+\binom{k}{2}+|J|}}{D'_{J}V_{J,J^{c}}}
\end{equation}
Since $\hat{J}$ was obtained from $J$ by replacing $y$ with $x$, we have \[|\hat{J}|=|J|-y+x\]
The difference between $|J|$ and $|\hat{J}|$ contributes a sign of $(-1)^{y-x}$.\\
Observe that on setting $a_{x}=a_{y}=z$, the denominators become the same polynomial. The term $V_{\hat{J}^{c}}$ acquires a sign of $(-1)^{y-x-(r-s)}$ when setting $a_{x}=a_{y}=z$, i.e 
\[V_{\hat{J}^{c}}|_{a_{x}=a_{y}=z}=(-1)^{y-x-(r-s)}V_{J^{c}}|_{a_{x}=a_{y}=z}\]
Recall that $N_{J}$ and $N_{\hat{J}}$ are both defined as a sum over permutations $\sigma\in S_{k}$. Let $\tau\in S_{k}$ be the permutation so that $(\tau(\hat{J}))_{s}=x$, and $(\tau(\hat{J}))_{\ell}=j_{\ell}\ \forall \ell\neq s$. This permutation when applied to $\hat{J}$ shifts the elements to the ``correct position" as in $J$, which means that every element of $\hat{J}$ appears in the same place as in $J$, and $x$ appears in the same place as $y$. This permutation is a cycle with sign $(-1)^{s-r-1}$.\\\\
This means that we can now sum over permutations $\sigma\tau\in S_{k}$ (where $\sigma$ varies) for the $N_{\hat{J}}$ term in the summand that corresponds to $\hat{J}$ in \eqref{twosummands}. Doing this will allow us to factor out $\left(\frac{N_{J}V_{J^{c}}(-1)^{k(n-k)+\binom{k}{2}+|J|}}{D'_{J}V_{J,J^{c}}}\right)|_{a_{x}=a_{y}=z}$ from both summands in \eqref{twosummands}.\\  
From our previous observations we can see that the total sign that is obtained on the summand corresponding to $\hat{J}$ after setting $a_{x}=a_{y}=z$ in \eqref{twosummands} is \[(-1)^{-y+x+y-x-(r-s)+s-r-1}=-1\]
Which means that after setting $a_{x}=a_{y}=z$, \eqref{twosummands} can be written as
\[\left(\left(\frac{N_{J}V_{J^{c}}(-1)^{k(n-k)+\binom{k}{2}+|J|}}{D'_{J}V_{J,J^{c}}}\right)|_{a_{x}=a_{y}=z}\right)\times(-1+1)=0\]
Pairing up the other fixed points with this property (i.e $y\in J,\ x\notin J,\ (J\backslash\{y\})\cup\{x\}\in Attr^{f}(I)$) in a similar manner and following this procedure allows us to conclude that the sum is 0.\\\\
The results in the previous 3 subsections show that setting $a_{x}=a_{y}=z$ for arbitrary $1\leq x<y\leq n$ gives us 0 for $\sum_{p_{J}\in X_{k,n}^{T}}\frac{N_{J}V_{J^{c}}(-1)^{k(n-k)+\binom{k}{2}+|J|}}{D'_{J}V_{J,J^{c}}^{\hbar}}$, which in turn allows us to conclude that $V$ divides this sum. Hence the required limit exists.
\end{proof}
\begin{corollary}\label{limit monomial}
    The limit of \ref{lim-exists} is a monomial in $\hbar$.
\end{corollary}
\begin{proof}Note that each summand of \ref{eq:thesum} is a ratio of homogeneous polynomials in $a_{1},\dots,a_{n},\hbar$, and the degrees of the numerators of the summands are equal (or some summands may vanish). The same holds for the degrees of the denominators of the summands. Hence
\[\frac{1}{V}\sum_{p_{J}\in X_{k,n}^{T}}\frac{N'_{J}V_{J}V_{J^{c}}(-1)^{kn+\binom{k}{2}+|J|}}{D'_{J}V_{J,J^{c}}^{\hbar}}=\frac{P(a_{1},\dots,a_{n},\hbar)}{Q(a_{1},\dots,a_{n},\hbar)}\]
where $P$ and $Q$ are homogeneous in $a_{1},\dots,a_{n},\hbar$ with no common irreducible factors. Since the limit 
\[\lim_{\mathbf{a}\to 0}\frac{P(a_{1},\dots,a_{n},\hbar)}{Q(a_{1},\dots,a_{n},\hbar)}\]
exists, we can write $P$ and $Q$ as polynomials in $\hbar$ with coefficients in $\mathbb{C}[a_{1},\dots,a_{n}]$. Hence in the limit $\mathbf{a}\to 0$, all lower order (in $\hbar$) terms of $P$ and $Q$ vanish, which implies that the limit is a monomial in $\hbar$.
\end{proof}
\newpage
\printbibliography
\hfill\\
\noindent
Matthew Crawford\\
University of North Carolina, Chapel Hill\\
Chapel Hill, NC\\
mbcc314@unc.edu
\\\\
Pavan Kartik\\
University of North Carolina, Chapel Hill\\
Chapel Hill, NC\\
pkartik1@unc.edu
\\\\
Reese Lance\\
University of North Carolina, Chapel Hill\\
Chapel Hill, NC\\
rlance@unc.edu
\end{document}